\documentclass[11pt,a4paper]{article}
\usepackage[margin=3cm]{geometry}

\usepackage[applemac]{inputenc}
\usepackage{appendix}
\usepackage{mathtools,placeins}
\usepackage{amsmath,amsfonts,amssymb}
\usepackage{amsthm,cleveref}
\usepackage{color}
\usepackage{graphicx}
\newtheorem{theorem}{Theorem}[section]
\newtheorem{lemma}[theorem]{Lemma}
\newtheorem{remark}[theorem]{Remark}

\newtheorem{proposition}[theorem]{Proposition}

\usepackage{enumerate}

\usepackage[applemac]{inputenc}

\usepackage{todonotes}

\newcommand{\est}{\bold{e}_{t}^{\bold{\bsig}}}
\newcommand{\eut}{\bold{e}_{tt}^{\bold{u}}}
\newcommand{\eup}{\bold{e}_{t}^{\bold{u}}}
\newcommand{\eu}{\bold{e}^{\bold{u}}}
\newcommand{\es}{\bold{e}^{\boldsymbol \sigma}}
\newcommand{\bw}{\bold{w}}
\newcommand{\bfe}{\bold{f}}

\newcommand{\bu}{\bold{u}}
\newcommand{\bv}{\bold{v}}
\newcommand{\bz}{\bold{z}}
\newcommand{\buh}{\bold{u}^h}

\newcommand{\bn}{\bold{n}}
\newcommand{\n}{\bold{n}}
\newcommand{\bx}{\bold{x}}
\newcommand\bg{\bold{g}}
\newcommand{\bff}{\bold{f}}

\newcommand{\bL}{\bold{L}}
\newcommand{\bV}{\bold{V}}

\newcommand{\bcsi}{\boldsymbol{\xi}}
\newcommand{\btau}{\boldsymbol{\tau}}

\newcommand{\bomegaI}{\boldsymbol{\omega}^I}
\newcommand{\bomegah}{\boldsymbol{\omega}^h}

\newcommand{\bS}{\boldsymbol{\Sigma}}
\newcommand{\taub}{\boldsymbol{\tau}}

\newcommand{\bsig}{\boldsymbol{\sigma}}
\newcommand{\bsigh}{\boldsymbol{\sigma}_h}
\newcommand{\beps}{\boldsymbol{\varepsilon}}

\newcommand{\calA}{{\mathcal A}} 
\newcommand{\calD}{{\mathcal D}} 
\newcommand{\calF}{{\mathcal F}} 
\newcommand{\calE}{{\mathcal E}}

\newcommand{\calM}{{\mathcal M}} 
\newcommand{\calT}{{\mathcal T}}

\def\IR{\mathbb R}

\def\IP{\mathbb P}
\def\IQ{\mathbb Q}

\def\A{a}
\newcommand{\R}{\mathbb{ R}}

\newcommand{\normsmall}[1]{\|  #1 \|}

\newcommand{\normA}[1]{\| #1 \|_{\A}}
\newcommand{\norm}[1]{\|  #1 \|}
\newcommand{\abs}[1]{\left|  #1 \right|}
\newcommand{\normtre}[1]{\|\!|  #1 \|\!|}

\newcommand{\normtreMDG}[1]{\|\!|  #1 \|\!|
_{\calE, \textrm{MDG}}}
\newcommand{\normtreIP}[1]{\|\!|  #1 \|\!|
_{\calE, \textrm{IP}}}

\newcommand{\normMDG}[1]{\|  #1 \|_{\calE,\textrm{MDG}}}

\newcommand{\normIP}[1]{\|  #1 \|_{\calE,\textrm{IP}}}

\newcommand{\tr}{\textrm{tr}}
\newcommand{\diver}[1]{\nabla \cdot #1}

\newcommand{\salto}[1]{[\![ #1]\!]}
\newcommand{\saltotre}[1]{\left[\!\left[\!\left[ #1 \right]\!\right]\!\right]}
\newcommand{\vect}[1]{\boldsymbol{#1}}\newcommand{\av}[1]{\{#1\}} %usata
 \newcommand{\K}{K} %usata
\renewcommand{\O}{\Omega} %usata
\newcommand{\Eh}{{{\mathcal F}_h}} %usata
\newcommand{\Eho}{{{\mathcal F}^{o}_h}} %usata
\newcommand{\EhoD}{{{\mathcal F}^{o}_h} \cup {{\mathcal F}^{D}_h} } 
\newcommand{\EhD}{{{\mathcal F}^{D}_h} } %usata
\newcommand{\EhN}{{{\mathcal F}^{N}_h} } %usata
\newcommand{\Th}{\mathcal{T}_h}
\newcommand{\Ehb}{{{\mathcal F}^{\partial}_h}} %usata

\newcommand{\V}{\vect{V}_h} 
\newcommand{\Dinf}{\mathrm{D}_{\ast}} 
\newcommand{\Dsup}{\mathrm{D}^{\ast}} 
\newcommand{\cuu}{{\bf c}_{11}} 
\newcommand{\cdd}{{\bf c}_{22}} 
\newcommand{\czz}{{\bf c}_{00}} 
\newcommand{\SF}{{\bf S}_F}
\newcommand{\rinf}{\rho_{\ast}} 
\newcommand{\rsup}{\rho^{\ast}}

\begin{document}
\title{Stability Analysis for Discontinuous Galerkin approximations of the elastodynamics problem}
%\runningtitle{Stability analysis for DG approximations of the elastodynamics problem}
\author{Paola F. Antonietti$^\flat$, Blanca Ayuso de Dios$^\S$, Ilario Mazzieri$^\flat$, 
and Alfio Quarteroni$^{\flat, \sharp}$}
\date{\today}

\maketitle 

%%%%%%%%%%%%%% ADDRESSES %%%%%%%%%%%
\begin{center}
{\small
$^\flat$ MOX, Dipartimento di Matematica, Politecnico di Milano\\
Piazza Leonardo da Vinci 32, I-20133 Milano, Italy.\\
{\tt paola.antonietti@polimi.it}, {\tt ilario.mazzieri@polimi.it}\\
\vskip 0.2cm
%%%%%%%%%%%%%%%%%%%%%%%%%%%%%%%%%
$^\S$ Centre de Recerca Matem\'{a}tica, Campus de Bellaterra, 08193 Bellaterra, Barcelona, Spain. {\tt bayuso@crm.cat}\\
\vskip 0.2cm
%%%%%%%%%%%%%%%%%%%%%%%%%%%%%%%%%
$^{\flat}$ CMCS, Ecole Polytechnique Federale de Lausanne (EPFL), Station 8, 1015 Lausanne, Switzerland. {\tt alfio.quarteroni@epfl.ch}
}
\end{center}
%%%%%%%%%%%%%%%%%%%%%%%%%%%%%%%%%
\vspace*{0.2cm}
\begin{abstract}
We consider semi-discrete discontinuous Galerkin approximations of a general elastodynamics problem, in both {\it displacement} and  {\it displacement-stress} formulations. We present the stability analysis of all the methods in the natural energy norm and derive  optimal a-priori error estimates. For the  displacement-stress formulation, schemes preserving  the total energy of the system are introduced and discussed. We include some numerical experiments in three dimensions to verify the theory.
\end{abstract}
%%%%%%%%%%%%%%%%%%%%%%%%%%%%%%%%%
\section{Introduction}\label{sec:intro}

Understanding the physics of earthquakes is of crucial importance for predicting their impacts on the human and natural environment and delineating seismic risk reduction strategies.
However, to produce realistic seismic events is  needed not only correct assumptions on the physical parameters of the model, but also numerical methods capable to to face the typical multi-scale nature of such problems. Nowadays, despite the the great development of High Performance Computer facilities,
the representation of seismic events is still a  challenging task that present two different  but fundamental difficulties: the need for {\it geometrical flexibility} , essential when dealing with complex wave phenomena, and  the {\it control of dissipation and dispersion errors}  introduced by the numerical scheme, since accurate approximation of amplitude and phase of the waves gives important informations about the interior structure and consistency of materials. These two pitfalls together with the fact  that the wavelengths of interest of a seismic event are usually small compared to the size of the body excited, imply that a large number of unknowns  is naturally involved in the computational model. Therefore, efficiency and scalability of the numerical methods (and suitable implementation of them) are mandatory on large parallel supercomputers in order to keep as low as possible the computational burden.
Discontinuous Galerkin~(DG) methods for elastodynamics have received  lot of attention in recent years, since they  
account for the complexity of the geometric constraints,  provide accurate solutions while keeping the computational effort as low as possible, are naturally oriented towards parallel computation,
see, \emph{e.g.}, \cite{RiWh03,DeBaSeWh08,WilcoxStadlerBursteddeGhattas2010,
AntoniettiMazzieriQuarteroniRapetti_2012} and the references therein. \\

The purpose of this paper is to design and analyze semi-discrete DG methods for a  general elastodynamic problem, considering both the {\it displacement } and the {\it displacement-stress} formulation.
So far,  two main streams have been followed in the design and analysis of DG methods: the {\it displacement} formulation and the {\it velocity-stress} formulation. For the former, DG methods of Interior Penalty type, symmetric and non-symmetric, have been proposed and analyzed in \cite{RiWh03}. The schemes are extended to Spectral-DG methods in \cite{AntoniettiMazzieriQuarteroniRapetti_2012} and to DG approximations of viscoelasticity in \cite{RSW07}. For  the {\it velocity-stress} formulations, the design of the DG methods follow the traditional guidelines in the construction of DG schemes for hyperbolic conservation laws. 
Conservative methods based on the use of central flux have been proposed in  \cite{Delcourteetal_2009}; while non-conservative methods based on upwinding fluxes are studied  \cite{KaserDumbser_2008}. The DG method developed in \cite{WilcoxStadlerBursteddeGhattas2010} is based on a velocity-strain formulation of the coupled elastic-acoustic wave equations; this allows the acoustic and elastic wave equations to be expressed in conservative form within the same framework.\\

In this paper we introduce a fairly general family of semidiscrete discontinuous Galerkin methods for a linear elastodynamic problem, considering DG methods for both the {\it displacement} and {\it displacement-stress} formulations of the problem.   The main goal is to identify the key ingredients to ensure stability for DG approximations of the general problem with mixed boundary conditions (that are typically encountered in the seismic applications).   Our stability analysis follow, as one should expect, the one for the continuous problem. For this reason, we start with the {\it displacement-stress} formulations which gives further insight on the features required by the methods.  Finite Element methods for the {\it displacement-stress} formulation were proposed and analyzed in the seminal work \cite{babis00}. Here, some extra difficulties arise in the analysis due to the discontinuous nature of the spaces and the fact that we consider the general problem with mixed boundary conditions.  However, the flexibility of DG framework allow us to construct  in a simple way, {\it displacement-stress} DG methods  that are fully conservative (in the sense that the total discrete energy is preserved).\\

For the {\it displacement} formulation, we consider Interior Penalty (IP) schemes, focusing on symmetric methods, similar to those considered for wave equation in \cite{GrScSc06}, but different from the IP schemes  introduced  in \cite{RiWh03,RSW07,AntoniettiMazzieriQuarteroniRapetti_2012} for linear elastodynamics. The IP methods considered in those works contain an extra penalization term that penalizes the time derivative of the displacement besides the displacement itself. Such extra penalization, whose physical meaning is completely unclear, was required to prove theoretically the stability of the resulting methods.  However, as we shall demonstrate via numerical experiments the inclusion of such extra term in the schemes,  seem to undermine the overall efficiency of the methods, since they impose a more restrictive condition on the time integration. Here,  we focus on symmetric IP methods, with no extra stabilization terms, and provide stability in the natural energy norm associated to the methods.\\

For all the DG methods considered, optimal error estimates are derived in a standard fashion, and the presented theory is verified throught three dimensional numerical tests. 
 We remark that the semidiscrete analysis that we present here, is an intermediate but fundamental step to derive the fully discrete stability analysis for the elastodynamics problem (when the coupled effect of the spatial and temporal discretization is taken into account for deriving \emph{a-priori} error estimates). This is out of the scope of the paper and will be subject for future work.  \\

The paper  is organized as follows. In in Section~\ref{section_model} we introduce the model problem and revise some key results.  The discrete notation is given in Section~\ref{sec:tools} and in Section \ref{DG-methods} we introduce the family of DG methods.
 The stability analysis is presented in Section~\ref{stability:sec}. A priori error estimate are derived in Section~\ref{sec:error_estimates}, and numerical experiments  are given in Section~\ref{sec:numerics}. The paper is closed with  Appendix~\ref{sec:proof-stability} and Appendix~ \ref{sec:ProofAuxiliary} containing some technical results.\\

{\sc Notation.} Throughout the paper, we use  standard notation for Sobolev
spaces~\cite{AdamsFournier2003}. For a bounded domain $D\subset
\mathbb{R}^{d}$, $d=2,3$,  we denote by $H^{m}(D)$ the $L^2$-Sobolev space of order~$m \geq 0$   and by  $\|\cdot \|_{m,D}$ and $| \cdot
|_{m,D}$ the usual Sobolev norms and seminorms, respectively. For $m
= 0$, we write $L^{2}(D)$ instead of $H^{0}(D)$. The space $H^1_{0, \Gamma}(D)$ is the subspace of $H^1(D)$ of functions 
with zero trace on $\Gamma \subseteq \partial D$.
Due to the nature of the problem, we only deal with vector--valued and matrix--valued functions and we use
boldface type for both. 
More precisely, the Sobolev spaces of vector--valued and symmetric tensor-valued functions are denoted by ${\bf H}^{m}(D)=[H^{m}(D)]^{d}$, and $ \vect{\mathcal H}^{m}(D)=[H^{m}(D)]^{d\times d}_{\rm{sym}}$, respectively. We will use $(\cdot\,,\cdot)_{D}$ to denote the standard inner product in any of the spaces ${\bf H}^{0}(D)=\bL^{2}(D)$ or $\vect{\mathcal H}^{0}(D)=\vect{\mathcal L}^{2}(D)$. Throughout the paper $C$ denotes a generic positive constant that may take different values in different places, but is always mesh independent. To avoid the proliferation of constants, we will use the notation $x \lesssim y$ to represent the inequality $x \leq C y$ for a constant as before. \\

For time dependent functions, we take the standard approach~\cite{AdamsFournier2003} of treating these as maps from a time interval $(0,T)$ into a Banach space $X$ and set
\begin{equation*}
\| \bv \|_{L^p(0,t;X)} = \left(\int_0^t \| \bv (\tau) \|_X^p  d\tau \right)^{1/p}, \quad 0 \leq t \leq T, \; 1 \leq p < \infty,
\end{equation*}
with the obvious modifications when $p =\infty$.

\section{Continuous problem}\label{section_model}
We consider an elastic medium occupying an open and bounded region $\Omega \subset \R^d$, $d=2, 3$, with 
Lipschitz boundary $\partial \Omega$ and outward normal unit vector $\bn $.
The medium is in a state of equilibrium under the influence of external forces, consisting in a volume force $\bff: \Omega \times [0,T] \rightarrow \R^d$, and a surface force ${\bf t}: \partial\Omega \times [0,T] \rightarrow \R^d$.
In linear elasticity it is possible to express the surface force ${\bf t}$ in term of Cauchy stress tensor $\bsig: \Omega \times [0,T] \rightarrow \mathbb{S}=\mathbb{R}_{\textrm{sym}}^{d \times d}$. Using the Cauchy stress formula ${\bf t} = \bsig \bn$, and the Gauss-Green theorem  we have that
\begin{equation}
\int_{\partial \Omega} {\bf t} \,dx 
= \int_{\partial \Omega} \bsig \bn  \,dx  
= \int_{\Omega}  \diver{\bsig} \,ds.
\end{equation}
The action of external loads induces on the body $\Omega$ a displacement vector field  that will be denoted by $\bu :\Omega \times [0,T] \longrightarrow \mathbb{R}^{d}$.  \\
Let the boundary $\partial \Omega$ be composed of two disjoint portions $\Gamma_{D}$, where the displacement vector $\bu$ is prescribed, and  $\Gamma_{N}$ where an external load $\bg$ applies. Assuming that $\textrm{meas}(\Gamma_D) >0$, 
the mathematical model of linear elastodynamics reads:
\begin{subequations}
 \begin{align}
\displaystyle \rho(\bx) \bu_{tt}(\bx, t)   -\diver{\bsig}(\bx, t)  & =  \bfe(\bx, t), 
&& \textrm{in} \; \Omega \times (0,T], \label{eqd:1} \\
\displaystyle  \calA \bsig(\bx, t) - \beps(\bu(\bx, t)) & = {\bf 0} 
&& \textrm{in} \; \Omega \times (0,T], \label{eqd:2}\\
\displaystyle \bu(\bx, t)  & =   \bold{0},  
&& \textrm{on} \; \Gamma_D\times (0,T], \label{eqd:3} \\
\displaystyle \bsig(\bx, t)  \bn(\bx)  & =  \bg(\bx, t),  
&& \textrm{on} \; \Gamma_N\times (0,T], \label{eqd:4}\\
\displaystyle \bu_t(\bx, 0)   & =  \bu_1(\bx), 
&& \textrm{in} \; \Omega \times \{ 0\}, \label{eqd:5} \\
\displaystyle \bu(\bx, 0)   &=  \bu_0(\bx),  
&&\textrm{in} \; \Omega \times \{ 0\}.\label{eqd:6}
 \end{align}
\end{subequations}
%\todo{Ho messo $(x,t)$ secondo me appesantisce e basta. Io lo leverei.}
A notation explanation follows.
To ease the reading here and in the following, we drop the explicit space/time dependence on the functions. 
The mass density $\rho \in L^{\infty}(\Omega)$ is a strictly positive function, i.e.,  
\begin{equation}\label{material:0}
0<\rinf \leq \rho(\bx) \leq \rsup
\qquad \forall \bx \in \Omega.
\end{equation}
We assume $\bff \in L^{2}((0,T]; \bL^2(\O)) $, $\bg \in {C}^1((0,T];{\bf H}^{1/2}(\Gamma_N))$, and suppose that the initial conditions $\bu_0$ and  $\bu_1$ for the displacement and the velocity field, respectively, are smooth enough functions, i.e.,  $\bu_0 \in {\bf H}^{1}_{0,\Gamma_D}(\O)$ and $\bu_1 \in \bL^{2}(\O)$.
Hereafter, we denote by $\beps(\bu)  : \Omega \longrightarrow \mathbb{S}$  the linearized strain tensor or symmetric gradient defined by 
\begin{equation*}
\beps(\bu)=  \frac{1}{2}(\nabla \bu +\nabla\bu^{\top}),
\end{equation*}
or, componentwise,
\begin{equation*}
 \varepsilon_{ij}(\bu)= \frac{1}{2}\left(\frac{\partial u_i}{\partial x_j}+\frac{\partial u_j}{\partial x_i} \right)\qquad i,j =1\ldots d.
 \end{equation*}

The compliance tensor $\calA=\calA(x) : \mathbb{S} \longrightarrow \mathbb{S}$ is  a bounded, symmetric and uniformly positive definite operator, encoding  the material properties,  such that
 \begin{equation}\label{defsigma}
\calA \bsig = \frac{1}{2\mu}\left( \bsig -\frac{\lambda}{3\lambda + 2\mu} \tr(\bsig) \mathbb{I} \right) \qquad \forall\, \bsig \in \mathbb{S}\;,
\end{equation}
where $\mathbb{I}\in \mathbb{R}^{d \times d}$ denotes the identity operator, $\tr(\cdot)$ stands for the trace operator
\begin{equation*}
\tr(\taub) = \sum_{i=1}^d \taub_{ii} \quad \forall \, \taub \in \mathbb{S},
\end{equation*}
and both  the Lam\'e parameters $\lambda, \mu \in L^{\infty}(\Omega)$  are  positive functions (isotropic case). 
Provided $\calA$ is invertible,  \eqref{defsigma}  is equivalent to the  Hooke's law
$\bsig=\calA^{-1}  \beps =  \calD  \beps $, with 
\begin{equation}\label{defeps}
\begin{aligned}
& \calD:\mathbb{S} \longrightarrow \mathbb{S},
&&\calD \btau =  {2\mu}\btau + \lambda \tr(\btau)\mathbb{I} \quad \forall \, \btau \in \mathbb{S}\;.
\end{aligned}
\end{equation}
In this case, from the properties of $\mathcal{A}$, it is directly inferred that $\calD$ satisfies the symmetry properties
\begin{equation*}
D_{ijk\ell} = D_{jik\ell} = D_{ij\ell k} = D_{k\ell ij} \qquad \forall \, i,j,k,\ell = 1,...,d,
\end{equation*}
and that it is also bounded and positive definite,  i.e., there exist $ \Dinf, \Dsup>0$ such that
\begin{equation}\label{cota:D}
\begin{aligned}
0 \, <\,  \Dinf (\btau , \btau )_{\Omega} 
\leq  ( \mathcal{D}  \btau , \btau)_{\Omega}  
\leq \Dsup ( \btau , \btau )_{\Omega}  
&& \forall\, \btau \in \mathbb{R}^{d\times d}, \btau \neq \bf{0}.
\end{aligned}
\end{equation}

To simplify the notation, in the following we will write $\bg_0=\bg(\bx, 0)$, $\bsig_0=\bsig(\bx, 0)= \calD \beps(\bu(\bx, 0))=
\calD \beps(\bu_0)$.\\

We next consider the variational formulation of \eqref{eqd:1}--\eqref{eqd:6}: for all $t \in (0,T]$ find $(\bu,\bsig) \in {\bf H}^{1}_{0,\Gamma_D}(\O)\times \vect{\mathcal L}^{2}(\O)$ such that:
\begin{subequations}
\begin{align}
 ( \rho \bu_{tt} , \bv )_{\O} + ( \bsig , \beps(\bv) )_{\O}  & =  ( \bfe , \bv )_{\O} + ( \bg , \bv)_{\Gamma_N} 
&& \forall\, \bv \in  {\bf H}^{1}_{0,\Gamma_D}(\O), \label{weak:1}\\
( \calA \bsig , \btau)_{\O} - ( \beps(\bu) , \btau )_{\O} & = 0  
&& \forall\, \btau \in \vect{\mathcal L}^{2}(\O).\label{weak:2}
\end{align}
\end{subequations}
Under the above regularity assumptions the saddle problem \eqref{weak:1}--\eqref{weak:2} has a unique solution $(\bu,\bsig) \in {\bf H}^{1}_{0,\Gamma_D}(\O)\times \vect{\mathcal L}^{2}(\O)$, \cite{DuvautLions_1976}, and satisfies \emph{a priori} stability estimate in the  energy norm
\begin{equation}\label{norm:E}
\| (\bu,\bsig) \|_{\calE}^{2} =  \| \rho^{1/2} \bu_{t} \|^2_{0,\O}  + \| \calA^{1/2} \bsig \|_{0,\O}^2.
\end{equation}
Since $\bsig=  \calA^{-1}  \beps$, cf. \eqref{eqd:2}, the last term on the right hand side can be replaced by $ \| \calD^{1/2} \beps(\bu) \|_{0,\O}^2$.

\begin{lemma}[\emph{A priori} stability estimate]\label{le:0}
Let $(\bu,\bsig) \in { C}^{1}((0,T]; {\bf H}^{1}(\O)) \times L^{2}((0,T]; \vect{\mathcal{L}}^{2}(\O)) $ be the solution of \eqref{eqd:1}--\eqref{eqd:6}.  Then,
\begin{enumerate}
\item If  $\bff=\bg = {\bf 0}$, the total energy of the system  is preserved in time
\begin{equation*}
\begin{aligned}
& \| (\bu(t),\bsig(t)) \|_{\calE} = \| (\bu_0,\bsig_0) \|_{\calE},
&& 0 <   t \leq T.
\end{aligned}
\end{equation*}
\item If $\bff \in L^{2}((0,T]; \bL^{2}(\O))$ and $\bg\in C^{1}((0,T]; {\bf H}^{1/2}(\Gamma_N))$, the following  \emph{a priori} energy estimate holds: 
\begin{equation*}
\begin{aligned}
&\| (\bu(t),\bsig(t)) \|_{\calE} \lesssim
\sqrt{\mathcal{G}} +  \int_0^t  \left( \rinf^{-1/2} \| \bfe(\tau) \|_{0,\O} +  \norm{\bg_{\tau}(\tau)}_{0,\Gamma_N} \right) \, d\tau
&& 0 <   t \leq T,
\end{aligned}
\end{equation*}
with 
\begin{equation*} 
\mathcal{G}= \|\bg_0\|_{0,\Gamma_N}^{2}
+(1+\Dinf^{-1}) \|(\bu_0,\bsig_0)\|^{2}_{\calE} \\
+  \Dinf^{-1} \sup_{0<t\leq T} \|\bg(t)\|_{0,\Gamma_N}^{2}.
\end{equation*}
being $\Dinf$, $\rinf$ the lower bounds on the stiffness operator $\mathcal{D}$ and on the mass density given in \eqref{cota:D} and \eqref{material:0}, respectively. 
\end{enumerate}
\end{lemma}
The proof is shown in Appendix \ref{sec:proof-stability}, cf. also \cite[Theorem 4.1]{DuvautLions_1976} for the general existence result. 
%\todo{Leverei questa frase.}
%Korn's first inequality \cite{DuvautLions_1976} is an essential tool to prove Lemma~\ref{le:0}.

\begin{remark}\label{rem:primal}
Choosing $\btau = \beps(\bv)$ in \eqref{weak:2} and substituting the result in \eqref{weak:1} it is possible to obtain the following equivalent weak problem: for all $t \in (0,T]$ find $\bu \in {\bf H}^{1}_{0,\Gamma_D}(\O)$ such that:
\begin{align}
 ( \rho \bu_{tt} , \bv )_{\O} + ( \calD \beps(\bu) , \beps(\bv) )_{\O}  & =  ( \bfe , \bv )_{\O} + ( \bg , \bv)_{\Gamma_N} & \quad \forall\, \bv \in  {\bf H}^{1}_{0,\Gamma_D}(\O). \label{weak:1p}
\end{align}
In particular, it is possible to prove that problem \eqref{weak:1p} is well posed and that its unique solution satisfies $\bu \in C((0,T]; {\bf H}^{1}_{0,\Gamma_D}(\O)) \cap C^1((0,T]; \vect{\mathcal L}^{2}(\O) )$.
\end{remark}

Finally, we recall a couple of mathematical tools needed in our forthcoming analysis.
We recall the following ``simplified'' version of the Gronwall's lemma, see \cite{Quarteroni_2009} for the proof.
\begin{lemma}[Gronwall's lemma]\label{Gronwall}
Let $\varphi \in L^1(0, T)$ be a positive function, $g$ a non negative constant, and $y$ a continuous non negative function in $(0, T)$ satisfying
\begin{equation*}
\begin{aligned}
&y^2(t) \leq g + \int_{0}^t \varphi(\tau)y(\tau) d\tau
&& \forall t \in (0, T).
\end{aligned}
\end{equation*}
Then
\begin{equation*}
\begin{aligned}
y(t) \leq \sqrt{g}+ \frac{1}{2}\int_{0}^t \varphi(\tau) d\tau  
&& \forall t \in (0, T).
\end{aligned}
\end{equation*}
\end{lemma}

%%%%%%%%%%%%%%
Finally, for $\bw, \bz$ regular enough, the
following integration by parts formula holds 
\begin{equation}\label{eq:ibp_formula}
\int_0^t(\bw, \bz_{\tau})d\tau
=(\bw(t), \bz(t))
-(\bw(0), \bz(0))
-\int_0^t(\bw_{\tau}, \bz)d\tau.
\end{equation}

\section{Notation and discrete technical tools}
\label{sec:tools}
In this section we introduce the basic notation and revise some technical tools that will be used in our analysis. \\

\noindent {\bf Mesh partitions.}
 We consider a family $\left\{\calT_{h}, 0<h\leq1\right\}$ of shape-regular conforming partitions of $\Omega$ into disjoint open elements  $\K$ such that $\overline{\Omega}=\cup_{\K \in \calT_{h}} \overline{\K}$, where 
each $\K \in \calT_{h}$ is the image of a fixed master element $\widehat{\K}$, \emph{i.e.}, 
$\K=F_\K(\widehat{\K})$, and $\widehat{\K}$ is either the open unit $d$-simplex or the open 
unit hypercube in $\mathbb{R}^d$, $d=2,3$. For a given mesh $\calT_h$, we define the mesh size of the partition as $h = \max_{\K \in \calT_h} h_{\K}$ with $h_{\K} = {\rm diam}(\K)$. 
Notice that the mesh may contain hanging nodes.
We collect all the interior (boundary, respectively) faces in the set $\calF_h^{o}$ ($\calF_h^{\partial}$, respectively) and set $\calF_h=\calF_h^{o}\cup \calF_h^{\partial}$. 
In particular $\calF_h^{\partial} = \calF_h^{D} \cup \calF_h^{N}$, where  $\calF_h^{D}=\calF_h^{\partial}\cap \Gamma_D$ and $\calF_h^{N}=\calF_h^{\partial}\cap \Gamma_N$  contain respectively all Dirichlet and Neumann boundary faces. Implicit in these definitions is the assumption that $\calT_h$ respect the decomposition of $\partial \Omega$ in the sense that any $F \in \calF_h^{\partial}$ belongs to the interior of exactly one of $ \calF_h^{D}$ or $ \calF_h^{N}$. 

An interior face (for $d=2$, ``face'' means ``edge'')  of $\calT_h$ is defined as the (non--empty) interior of $\partial \K^+\cap \partial \K^-$,
where $\K^+$ and $\K^-$ are two adjacent elements of $\calT_h$. Similarly, a boundary face of $\calT_h$ is defined as the (non-empty) interior of
$\partial \K \cap \Omega$, where $\K$ is a boundary element of $\calT_h$.
 We also assume that for all  $\K \in \calT_h$ and for all $F \in \calF_h$, 
$h_{\K} \lesssim h_F$, where $h_F$ is the diameter of $F \in \calF_h$. This last assumption implies that the 
maximum number of hanging nodes on each face is uniformly bounded.

Finally, we assume that a {\em bounded local variation} property holds: for any pair of elements $\K^+$ and $\K^-$ sharing a $(d-1)$--dimensional face $h_{\K^+} \approx h_{\K^-}$, see~\cite{GeorgoulisHallHouston2007b}, for example.\\

For $s\geq 1$, we define the \emph{broken} Sobolev spaces
\begin{align*}
{\bf H}^{s}(\Th)&=\left\{ \bv \in
    \bL^{2}(\O)~\mbox{such that}~\bv\big|_{\K} \in
    {\bf H}^{s}(\K), \quad \forall\, \K\in \Th\,\right\}, 	\\
\vect{\mathcal{H}}^{s}(\Th)&=\left\{ \btau \in
    \vect{\mathcal{L}}^{2}(\O)~\mbox{such that}~\btau\big|_{\K} \in
    \vect{\mathcal{H}}^{s}(\K), \quad \forall\, \K\in \Th\,\right\}.
\end{align*}
We will also denote by $(\cdot \,,\cdot)_{\Th}$ and $\langle \cdot\, , \cdot \rangle_{\Eh}$ the $\bL^{2}(\Th)$ and $\bL^{2}(\Eh)$ inner products, respectively, and use the convention that
\begin{equation*}
\begin{aligned}
&(\boldsymbol{\varphi}, \boldsymbol{\psi})_{\Th} =\sum_{K \in \Th} (\boldsymbol{\varphi}, \boldsymbol{\psi})_K
&& 
\langle \boldsymbol{\varphi}, \boldsymbol{\psi} \rangle_{\Eh}=\sum_{F \in \Eh} (\boldsymbol{\varphi}, \boldsymbol{\psi})_F.
\end{aligned}
\end{equation*}
The same notation will be used for the $\vect{\mathcal{L}}^{2}(\Th)$ and $\vect{\mathcal{L}}^{2}(\Eh)$ inner products.\\

\noindent {\bf Trace operators.} Let $F\in \Eho$ be an interior face shared by two elements of $\calT_h$, say $\K^+$ and $\K^-$ and let $\bn^+$ and $\bn^-$ denote the normal unit vectors on $F$ pointing outward $\K^+$ and $\K^-$, respectively. For a vector  $\bv\in {\bf H}^{1}(\Th)$  we  denote by $\bv^+$  and $\bv^-$ the  traces of $\bv$ on $F$ taken within the interior of $\K^+$ and $\K^-$, respectively. 
For $\btau \in \vect{\mathcal{L}}^{2}(\Th)$, $\btau^+$ and $\btau^-$ are defined analogously.
For a scalar $\delta \in [0,1]$, the \emph{weighted average} of a vector $\bv$ and a tensor $\btau$ are defined in the usual way
\begin{align*}
&\av{\bv}_{\delta}=\delta \bv^+ + (1-\delta)\bv^-,  
&\av{\btau}_{\delta}=\delta \btau^+ + (1-\delta)\btau^-, 
&&\forall F\in \Eho.
\end{align*}
Whenever $\delta=1/2$ we neglect the subscript and simply write $\av{\cdot}$, since the {\it weighted average} reduces to the standard average.
On boundary faces $F \in \Ehb$,  we set
\begin{align*} 
& \av{\bv} = \av{\bv}_{\delta} = \bv,
& \av{\btau} = \av{\btau}_{\delta} = \btau. 
\end{align*}
To define the jump of a vector
 $\bv\in {\bf H}^{1}(\Th)$, we first define $\bv \odot \n=(\bv \n^{T}+\n \bv^{T})/2$, which is nothing but the symmetric part of the tensor product  $\bv \times \n$. Then, following \cite{ArnoldBrezziFalkMarini_2007}, we set
 \begin{equation}\label{jump:vector}
\begin{aligned}
\salto{\bv}&=\bv^+\odot\n^+ + \bv^-\odot\n^- 
&& \forall F\in \Eho,
&&\salto{\bv}&=\bv\odot\n 
&&\forall F\in \Ehb.
\end{aligned}
\end{equation}
Notice that with this definition $\salto{\bv}$ is a symmetric tensor, i.e.,  $\salto{\bv} \in \mathbb{S}$. 

For a symmetric tensor $\btau \in \vect{\mathcal{L}}^{2}(\Th)$ we define
 \begin{equation*}
\begin{aligned}
&\salto{\btau }=\btau^+\, \n^+ + \btau^- \, \n^-
&& \forall F\in \Eho\;,
\end{aligned}
\end{equation*}
and observe that $\salto{\btau}$ is a vector--valued function. \\

Denoting by $\n_{\K}$ the outward unit normal to $\partial\K$, we shall frequently use the following identity that can be easily checked
\begin{equation}\label{per-tens0}
\sum_{\K\in \Th}  \langle \taub \,   \n_{\K}, \bv \rangle_{\partial \K} 
=\sum_{\K\in \Th} \langle \bv \times \n_{\K} ,  \btau  \rangle_{\partial \K}
= \langle \av{\taub} , \salto{\bv}\rangle_{\Eh} 
+ \langle \salto{\taub}  ,\av{\bv}\rangle_{\Eho}
\end{equation}
for all $\taub  \in \vect{\mathcal L}^{2}(\Th)$ and for all $\bv \in {\bf H}^{1}(\Th)$.
Note  that it also follows 
\begin{equation}\label{eq24}
\begin{aligned}
&\langle \bu^{\pm} , \av{\btau}_{\delta} \n^{\pm} \rangle_{\Eh} 
=  \langle \av{\btau}_{\delta}, \bu^{\pm} \odot \n^{\pm}  \rangle_{\Eh} && \forall \,\delta \in [0,1].
\end{aligned}
 \end{equation}

The \textit{weighted average} of tensors  will be  occasionally expressed as a linear combination of the jump and the standard average operators. In fact, it can be checked that 
\begin{equation}\label{eq:tensor}
\begin{aligned}
\av{\btau}_{\delta} \n^{+} = \av{\btau} \n^{+}  +\frac{(2\delta -1)}{2} \salto{\btau} 
&& \forall \, \delta \in [0,1]
&& \forall F\in \Eho.
\end{aligned}
\end{equation}
Combining now \eqref{eq:tensor} and \eqref{eq24}, one obtains
\begin{align*}
- \langle  \av{\bu^h}_{(1-\delta)} - \av{\bu^h} , \salto{\calD\beps(\bu^h)}\rangle_{\Eho} 
&=  (\delta-1/2)\langle \salto{\calD\beps(\bu^h)},(\bu^{h})^{+}-(\bu^{h})^{-} \rangle_{\Eho} \\
&= \langle  \av{\calD \beps(\bu^{h})}_\delta\,\n^{+} -  \av{\calD \beps(\bu^{h})} \,\n^{+},(\bu^{h})^{+}-(\bu^{h})^{-} \rangle_{\Eho} \\
&= \langle \av{\calD \beps(\bu^{h})}_\delta -  \av{\calD \beps(\bu^{h})},(\bu^{h})^{+} \odot\n^{+} +(\bu^{h})^{-} \odot \n^{-} \rangle_{\Eho}\;,
\end{align*}
which leads to the following identity that will be used often in this work:
\begin{equation}\label{eq:media-mediad}
- \langle  \av{\bu^h}_{(1-\delta)} - \av{\bu^h} , \salto{\calD\beps(\bu^h)}\rangle_{\Eho} 
=\langle \av{\calD \beps(\bu^{h})}_\delta-\av{\calD \beps(\bu^{h})}, \salto{\bu^{h}}\rangle_{\Eho}.
\end{equation}

\begin{remark}\label{rem:def-jump}
In \cite{RiShWhWh03,AntoniettiMazzieriQuarteroniRapetti_2012} the authors considered the following definition of the jump of vector--valued functions,  
\begin{equation*}
\begin{aligned}
\saltotre{\bv}=\bv^+\times \n^+ + \bv^-\times \n^- 
&&\forall F\in \Eho, 
&&\saltotre{\bv}= \bv\times \n  
&&\forall F\in \Ehb\;,
\end{aligned}
\end{equation*}
which is slightly different to that given in \eqref{jump:vector} and considered here. Observe that with the above definition, $\saltotre{\bv}$  is still a tensor but  it is not necessarily symmetric.  Notice though, that for any $\taub \in \vect{\mathcal H}^{1}(\Th)$ the following identity holds
\begin{equation*}
\langle \saltotre{\bv} , \av{\taub}\rangle_{\Eho}=\langle \salto{\bv} , \av{\taub}\rangle_{\Eho}\;.
\end{equation*}
\end{remark}
%%%%%%%%%%%%

\noindent {\bf Finite element spaces.}
For  $k\geq 1$ we define the finite element spaces $\V$ and $\bS_h$ as
\begin{align*}
\V
 & =\{ \bu \in \bL^2(\Omega)\ :\ \bu \circ F_\K \in [\calM^{k}(\widehat{\K})]^d \quad \forall\  \K \in \calT_h\},\\
\bS_h & =  \{ \btau \in \vect{\mathcal{L}}^2(\Omega)\ :\ \btau \circ F_\K \in [\calM^{k}(\widehat{\K})]^{d\times d} \quad \forall\ K \in \calT_h \},
\end{align*}
where $\calM^{k}(\widehat{\K})$ is either the space $\IP^{k}(\widehat{\K})$ of polynomials of degree at most $k$ on $\widehat{\K}$, if $\widehat{\K}$ is the reference $d$-simplex, or the space $\IQ^{k}(\widehat{\K})$ of tensor--product polynomials on $\widehat{\K}$ of degree $k$ in each coordinate direction, if $\widehat{\K}$ is the unit reference hypercube in $\IR^d$.\\

%
%\COMMENT{{\rosso Veramente non capisco la raggione del remark, perche io avevo lasciato i gradi diversi per vedere cosa succedeva nella stima dell'errore... ma credo non ci siamo capiti.... Per me va tolto, perche no dici niente di piu...}}
%{\rosso
%\begin{remark}\label{rem:gradi:polinom}
%For the sake of simplicity, we consider the same polynomial approximation order for the displacement and the linearized strain tensor. Other choices of approximation spaces are possible. For example, any pair of spaces $(\V,\bS_h)$ such that $\beps(\V) \subseteq \bS_h$ can be used and the corresponding analysis can be carried out following the same lines presented here. Whenever such inclusion is not satisfied a suitable projection operator has to be introduced, cf. \emph{e.g.}, \cite{BrezziHughesMariniMasud_2005,AntoniettiHeltai_2007}. 
%\end{remark}
%E poi non vedo neccesita di citare i due lavori...??}
%
%
\noindent {\bf Technical tools.}
We recall some results that will be used in our the analysis.
%Assuming that meas$(\Gamma_D)>0$, we recall the discrete Poincare inequality (cf. \cite{brennerP})
%\begin{equation*}
%\sum_{K \in \Th}\norm{ \bv}_{0,\K}^{2} \lesssim
%\sum_{K \in \Th} \norm{\nabla_h \bv}_{0,\K}^{2} 
%+ \sum_{F \in \calF_h^{o} \cup \calF_h^{D}} \|h_F^{-1/2}\salto{\bv}\|_{0,F}^{2}, \quad \forall\, \bv \in {\bf H}^1(\Th)\;.
%\end{equation*}
% and, the discrete Korn's inequalities for  ${\bf H}^1(\Th)$-vector fields (see \cite{Br04}):
%\begin{equation*}
%\sum_{K \in \Th}\norm{\nabla_h \bv}_{0,T}^{2} \lesssim
%\sum_{K \in \Th} \| \beps(\bv) \|^2_{0,K} 
%+ \sum_{F \in \calF_h^{o} \cup \calF_h^{D}} \|h_F^{-1/2}\salto{\bv}\|_{0,F}^{2}, \quad \forall\, \bv \in {\bf H}^1(\Th)\;.
%\end{equation*}
Agmon's and trace inequalities valid for any $\bv \in {\bf H}^{1}(\K)$ read as
\begin{subequations}
\begin{align}
\| \bv \|_{0,F} &\lesssim h_K^{-1} \| \bv \|^2_{0,K} + h_K | \bv |_{1,K}^2,
\label{agmon}\\
 h \|\bv \|^{2}_{0,F} & \lesssim \| \bv \|_{1,\K}^2,
\label{trace:0}
\end{align}
\end{subequations} 
for any $F\in \Eh$, $F\subset \partial \K$. We will also use the $L^{p}$-version of the above trace inequality,  which holds for all $\bv \in \vect{W}^{1,p}(\K)$ and reads
\begin{equation}\label{trace:0p}
\begin{aligned}
&h^{1/p} \|\bv \|_{L^{p}(F)} \lesssim  \| \bv \|_{L^{p}(\K)} 
&& \forall\, F\in \Eh, 
&& F\subset \partial \K,
&& 1\leq p\leq \infty\;.
\end{aligned}
\end{equation}
In the above inequalities, the hidden constants are independent of the mesh size but (when applied to discrete polynomials) might depend on the polynomial degree.\\

For discrete functions (scalar, vector and tensor), we will also frequently use the following well known inequalities (see \cite{B-ciarlet} for details and proofs):
let $\omega$ be either an element, an edge or a face of the decomposition $\Th$, and let $\bv$ be a polynomial of degree $k\geq 1$ over $\omega$, then
\begin{equation}\label{ine:00}
\begin{aligned}
&\|\bv \|_{L^{p}(\omega)} \lesssim \mbox{meas}(\omega)^{\left(\frac{1}{p}-\frac{1}{q}\right)} \|\bv \|_{L^{q}(\omega)} && 1\leq p, q\leq \infty \;.
\end{aligned}
\end{equation}
For any $K\in\Th$ and any $F\subset \partial\K$ inverse inequality can be written as\begin{equation}
\label{eq:inverse}
\begin{aligned}
|\bv |_{m, K} 
&\lesssim h_K^{s-m} |\bv |_{s, K} 
&& s\leq m 
&& \forall\, \bv \in \V\;,
\end{aligned}
\end{equation}
where the constants in \eqref{ine:00} and \eqref{eq:inverse} are independent of the mesh size but depend on the polynomial degree $k$.  \\
Notice that in \eqref{eq:inverse} we have also used the inverse inequality.

\section{Discontinuous Galerkin approximations}\label{DG-methods}
In this section, we introduce the family of semidiscrete DG  approximations to  \eqref{eqd:1}--\eqref{eqd:6} that we consider in this work. \\
The derivation of the methods follows  closely \cite{ABCM}, but with a slight difference when introducing the schemes for the displacement formulation. 
We start by considering a  general variational formulation for DG methods:
Find $(\buh ,\bsig^h) \in C^{2}([0,T];\V)\times C^{0}([0,T];\bS_h)$ such that
\begin{align*}
&(\rho \buh_{tt}, \bv )_{\Th} +(\bsig^h , \beps( \bv ))_{\Th}   -\langle \av{\widehat{\bsig}}, \salto{\bv}\rangle_{\Eh} -\langle \salto{\widehat{\bsig}}, \av{\bv}\rangle_{\Eho}=(\bfe, \bv)_{\Th} && \forall\, \bv \in \V,\\
&( \calA \bsig^h , \btau)_{\Th} -( \beps(\bu^h), \btau)_{\Th}-\langle \av{\widehat{\bu}-\bu^h}, \salto{\btau}\rangle_{\Eho}  -\langle \salto{\widehat{\bu}-\bu^h}, \av{\btau}\rangle_{\Eh}  =0 &&\forall\, \btau \in \bS_h,
\end{align*}
where $(\widehat{\bu},\widehat{\bsig}):=\left(\widehat{\bu}(\bu^h,\bsig^h), \widehat{\bsig}(\bu^h, \bsig^h)\right) :\left({\bf H}^{1}(\Th)\times \vect{\mathcal{H}}^1(\Th)\right)^{2}\longrightarrow (\bL^{2}(\Eh),\vect{\mathcal{L}}^{2}(\Eh))$ are the numerical fluxes  that will be defined later on and identify the corresponding DG method. On boundary faces $F\in \Ehb$ we always define the numerical fluxes according to the boundary conditions \eqref{eqd:3}-\eqref{eqd:4}:
\begin{equation*}
\begin{aligned}
 \widehat{\bu}&={\bf 0} 
&& \textrm{ on }  F \in \EhD,
&\widehat{\bu}& ={\bf u}^{h} -\cdd (\bsig^{h} \bn - \bg)
&& \textrm{ on }  F \in \EhN,\\
\widehat{\bsig} \,\bn &=\bsig^{h}\bn -\cuu \bu^{h}
&& \textrm{ on }  F \in \EhD,
&\widehat{\bsig} \,\bn & =  \bg
&& \textrm{ on }  F \in \EhN.\\
\end{aligned}
\end{equation*}
Here, $\cuu$ and $\cdd$ are functions (possibly equal to zero) that we will choose later on.  Then, the DG formulation becomes:
Find $(\bu^{h},\bsig^{h}) \in C^{2}([0,T];\V)\times C^{0}([0,T];\bS_h)$ such that
\begin{subequations}
\label{dg:3}
\begin{align}
&(\rho {\bu}^h_{tt}, \bv )_{\Th} +(\bsig^h , \beps( \bv ))_{\Th}   -\langle \av{\widehat{\bsig}}, \salto{\bv}\rangle_{\Eho}  -\langle \salto{\widehat{\bsig}\, }, \av{\bv}\rangle_{\Eho}&& \label{dg:3a}\\
&\qquad\qquad\qquad+\langle \cuu \bu^{h} , \bv\rangle_{ \EhD} -\langle \bsig^{h} \n, \bv\rangle_{ \EhD} =(\bfe, \bv)_{\Th} +\langle \bg, \bv \rangle_{\calF_h^N} &&\forall\, \bv \in \V,\notag \\
&( \calA \bsig^h , \btau)_{\Th} -( \beps(\bu^h), \btau)_{\Th}-\langle \av{\widehat{\bu}-\bu^h}, \salto{\btau}\rangle_{\Eho} - \langle \salto{\widehat{\bu}-\bu^h}, \av{\btau}\rangle_{\Eho} && \label{dg:3b}\\
&\qquad\qquad \qquad+ \langle \salto{\bu^h}, \av{\btau}\rangle_{\EhD}
+\langle \cdd (\bsig^{h}\bn -\bg), \btau\,\n\rangle_{\EhN} 
=0 \qquad &&\forall\, \btau \in \bS_h\notag.
\end{align}
\end{subequations}

\subsection{DG methods for the  displacement-stress formulation} 
\label{sec:mixedDG}

We present now several methods for approximating the  displacement-stress formulation, by selecting different choices of the numerical fluxes in \eqref{dg:3}.
We restrict our attention to methods for which the numerical fluxes $\widehat{\bu}$ and $\widehat{\bsig}$ are singled valued. As a consequence $\salto{\widehat{\bu}}={\bf 0}$ and $\salto{\widehat{\bsig}}={\bf 0}$ on internal faces. \\
Now, in analogy with the method introduced in \cite{CastilloCockburnPerugiaSchotzau_2000} for second order elliptic problems, the full DG (FDG) approximation is characterized by the choices
 \begin{equation}\label{flux:dg}
\begin{aligned}
&\widehat{\bu} = \av{\bu^{h}}_{1-\delta}    - \cdd \salto{\bsig^{h}},
&&\widehat{\bsig} = \av{\bsig^{h}}_{\delta}  - \cuu \salto{\bu^{h}},
&& F \in \Eho,
\end{aligned}
\end{equation}
where
 \begin{equation}\label{las-ces}
\begin{aligned}
& \cuu=c_1 h_F^{-1}k^{2} \av{\calD}
&& \cdd = c_2 h_F k^{-2} \av{\calD}^{-1}
&& F \in \Eho.
\end{aligned}
\end{equation}
Here $c_1,c_2\geq 0$ are constants (sometimes required to be strictly positive).  On boundary faces,  $\cuu$ and $\cdd$ are defined accordingly. Substituting \eqref{flux:dg}  into  \eqref{dg:3},  we get:
\begin{equation}\label{dg:mixform}
\begin{aligned}
&(\rho {\bu}^h_{tt}, \bv )_{\Th} +(\bsig^h , \beps( \bv ))_{\Th}   -\langle \av{\bsig^h}_{\delta}, \salto{\bv}\rangle_{\Eho}  +\langle \cuu \salto{\bu^{h}}, \salto{\bv}\rangle_{\Eho}&&\\
&\qquad\qquad\qquad+\langle \cuu \bu^{h} , \bv\rangle_{ \EhD} -\langle \bsig^{h} \n, \bv\rangle_{ \EhD} =(\bfe, \bv)_{\Th} +\langle \bg, \bv \rangle_{\calF_h^N} &&\forall\, \bv \in \V,\\
&( \calA \bsig^h , \btau)_{\Th} -( \beps(\bu^h), \btau)_{\Th}-\langle \av{\bu^{h}}_{(1-\delta)}-\av{\bu^h}, \salto{\btau}\rangle_{\Eho} +\langle \cdd\salto{\bsig^{h}}, \salto{\btau}\rangle_{\Eho}  &&\\
&\qquad\qquad \qquad+ \langle \salto{\bu^h}, \av{\btau}\rangle_{\Eho\cup \EhD}
+\langle \cdd (\bsig^{h}\bn -\bg), \btau\,\n\rangle_{\EhN} 
=0 \qquad &&\forall\, \btau \in \bS_h.
\end{aligned}
\end{equation}
Special cases are the \emph{local discontinuous Galerkin} (LDG) method and the \emph{alternating choice of fluxes} (ALT) methods. The former is characterized by setting $\cdd= \bf{0}$, whereas the latter by $\cdd= \cuu= {\bf 0}$ and $\delta=1$ or $\delta=0$. For $\delta=1$ the numerical fluxes become
 \begin{equation}\label{flux:alt}
\widehat{\bu}=(\bu^{h})^{-}, \qquad \quad \widehat{\bsig}=(\bsig^{h})^{+} \,.
  \end{equation}
This choice has been frequently used to design DG approximation for time dependent problems with high order derivatives \cite{shu00,shu0}. To
 our knowledge, the ALT method has never been considered for the elastodynamics problem. In the next section, we will show that stability for this method can be guaranteed only in the case of Dirichlet-type boundary conditions (or periodic boundary conditions, generally used in  \cite{shu00,shu0}, but not realistic in the present context).
\subsection{DG methods for the displacement formulation} 
\label{sec:DGprimal}
We now consider DG methods in displacement formulation that could be regarded as a direct approximation to \eqref{weak:1p}, and so only the displacement $\bu$ is discretized. To obtain the variational formulation starting from  \eqref{dg:3}, the numerical flux $\widehat{\bsig}$ is defined as a function of $\bu^{h}$ only, and the discrete stress tensor $\bsig^{h}$ is eliminated (by setting $\btau=\mathcal{D}\epsilon(\bv)$ in \eqref{dg:3b} and combining the resulting equation with  \eqref{dg:3a}). To allow for such elimination, is implicitly assumed that the finite element spaces $(\V,\bS_h)$ are such that $\beps(\V) \subseteq \bS_h$.  

The definition of the numerical fluxes on boundary faces has to be modified taking into account that $\bsig^{h}\n=\calD\beps(\bu^{h})\n$ on $\Ehb$ (and $\cdd\equiv {\bf 0}$ now). Hence, we have: 
 \begin{equation*}
\begin{aligned}
 \widehat{\bu}&={\bf 0} 
&& \textrm{ on }  F \in \EhD,
&\widehat{\bu}& ={\bf u}^{h} 
&& \textrm{ on }  F \in \EhN,\\
\widehat{\bsig} \,\bn &=\calD\beps(\bu^{h})\n-\SF \bu^{h}
&& \textrm{ on }  F \in \EhD,
&\widehat{\bsig} \,\bn & =  \bg
&& \textrm{ on }  F \in \EhN,\\
\end{aligned}
\end{equation*}
where, to be consistent, we have replaced the parameter $\cuu$ by $\SF$, which plays the same role and scales in the same way (see below for its precise definition), 
but, differently from  $\cuu$,  will undergo to a technical restriction.

For $\delta \in [0,1]$, we specify $\widehat{\bsig}$ as follows %(this definition is general enough to be applicable to most DG methods for the displacement formulation):
\begin{equation}\label{def:sig0}
\widehat{\bsig}=
\left\{
\begin{aligned}
&\av{ \calD \beps(\bu^h)}_{\delta} -\SF \salto{\bu^h}
&& F\in \Eho,  \\
&\calD \beps(\bu^h) -\SF \bu^h \bn
&& F\in \EhD,  
\end{aligned}
\right.
\end{equation}
where
\begin{equation}\label{def:penalty_SIP}
\begin{aligned}
\SF=\czz h_F^{-1}k^{2}\av{\calD}   
&& \forall\, F\in \Eho\cup \EhD,
\end{aligned}
\end{equation}
and $\czz$ is a strictly positive constant that has to be chosen sufficiently large, see below.

By setting now  $\btau = \calD \beps(\bv) \in \bS_h$ in the \eqref{dg:3b} we find for all $\bv \in \V$:
\begin{align*}
( \calA \bsig^h , \calD \beps(\bv))_{\Th} = ( \beps(\bu^h), \calD \beps(\bv))_{\Th}+ \langle \av{\widehat{\bu}-\bu^h}, \salto{\calD \beps(\bv)}\rangle_{\Eho} 
  +\langle \salto{\widehat{\bu}-\bu^h}, \av{\calD \beps(\bv)}\rangle_{\Eho\, \cup \,\calF_h^D}\;.
\end{align*}
Since $\calA$ is symmetric and positive definite it holds
\begin{equation*}
\begin{aligned}
&( \calA \bsig^h , \calD \beps(\bv))_{\Th} = (  \bsig^h , \calA^\top\calD \beps(\bv))_{\Th} = (\bsig^h,\beps(\bv))_{\Th},
&& \forall \, \bv \in \V,
\end{aligned}
\end{equation*}
and so,
\begin{equation*}
(\bsig^h,\beps(\bv))_{\Th} = ( \beps(\bu^h), \calD \beps(\bv))_{\Th}+\langle \av{\widehat{\bu}-\bu^h}, \salto{\calD \beps(\bv)}\rangle_{\Eho}
+\langle \salto{\widehat{\bu}-\bu^h}, \av{\calD \beps(\bv)}\rangle_{\Eho\, \cup \,\EhD}.
\end{equation*}
Combining now the above equation together with \eqref{dg:3a} and the definition of numerical flux $\widehat{\bsig}$ given in \eqref{def:sig0}, we finally get the following formulation: 
 Find  $\bu^{h} \in C^{2}([0,T];\V)$ such that 
\begin{multline*}
(\rho {\bu}^h_{tt}, \bv )_{\Th} 
+( \beps(\bu^h), \calD \beps(\bv))_{\Th} 
+\langle \av{\widehat{\bu}-\bu^h}, \salto{\calD \beps(\bv)}\rangle_{\Eho}  \\ 
+\langle \salto{\widehat{\bu}-\bu^h}, \av{\calD \beps(\bv)}\rangle_{\Eho\, \cup \,\calF_h^D} 
-\langle \av{\calD \beps(\bu^{h})}_{\delta}, \salto{\bv}\rangle_{\Eho\, \cup \, \calF_h^D}  \\ 
+  \langle \SF\salto{ \bu^{h}}, \salto{\bv}\rangle_{\Eho\, \cup \, \calF_h^D}
=(\bfe, \bv)_{\Th} 
+\langle \bg, \bv \rangle_{\calF_h^N}, \qquad \forall\, \bv \in \V,
\end{multline*}
which corresponds to the family of classical Interior Penalty (IP) methods. We focus on the symmetric IP, but for completeness we describe the corresponding non-symmetric version.
\subsubsection*{Weighted symmetric interior penalty method.} Following \cite{Stenberg_1998}, to obtain the weighted Symmetric Interior Penalty method (SIP($\delta$))  we define
\begin{equation*}
\begin{aligned}
\widehat{\bu} = \av{\bu^{h}}_{1-\delta}  
&&  \forall \delta \,\in [0,1].
\end{aligned}
\end{equation*}
For  $\delta=1/2$, $\widehat{\bu}=\av{\bu^{h}}$, we get the classical Symmetric Interior Penalty (SIP) method  \cite{arnold_1982}.
%%%%%%%%%%%%%
\subsubsection*{Non-symmetric and incomplete interior penalty methods.} We take $\delta=1/2$ and define the numerical flux as 
\begin{equation*}
\begin{aligned}
&\widehat{\bu} = \av{\bu^{h}} +\frac{(1+\theta)}{2} \salto{\bu^{h}}\,\n_K, 
\end{aligned}
\end{equation*}
where $\n_{\K}$ is the outward unit normal vector to the element $\K \in \Th$.  For $\theta=1, 0$ we obtain the Non-symmetric Interior Penalty (NIP) and Incomplete Interior Penalty (IIP) methods, respectively.
Observe that for $\theta\ne -1$, the numerical flux is not singled valued. Indeed, as
$\salto{\n_{\K}}=2$, it can be easily checked that $\salto{\widehat{\bu}}= (1+\theta) \salto{\bu^{h}}$.\\

Notice that  all the IP  schemes can be recast in the following variational formulation: Find $\bu^{h} \in C^{2}([0,T];\V)$ such that
\begin{equation}\label{dg:ip3}
(\rho {\bu}^h_{tt}, \bv )_{\Th}   
+ \A(\bu^h, \bv) =(\bfe, \bv)_{\Th} +\langle \bg, \bv \rangle_{\calF_h^N}\quad \forall\, \bv \in \V.
\end{equation}
with $\theta$ defined as before, and
$a(\cdot, \cdot): \V\times \V \longrightarrow \IR$ given by
 \begin{multline}\label{eq:defa}
\A(\bw, \bv)= ( \beps(\bu^h), \calD \beps(\bv))_{\Th}-\langle \av{\calD \beps(\bw)}_{\delta}, \salto{\bv}\rangle_{\Eho\, \cup \, \calF_h^D}\\
+\theta \langle \salto{\bw}, \av{\calD \beps(\bv)}_{\delta} \rangle_{\Eho\, \cup \,\calF_h^D} 
+  \langle \SF\salto{\bw}, \salto{\bv}\rangle_{\Eho\, \cup \, \calF_h^D}.
\end{multline}
In \cite{RiShWhWh03,AntoniettiMazzieriQuarteroniRapetti_2012} the authors consider a variant of the above DG discretization; namely
\begin{equation}\label{extraAA}
(\rho {\bu}^h_{tt}, \bv )_{\Th}   
+ \A(\bu^h, \bv)  +  \langle c_F\salto{\buh_t}, \salto{\bv}\rangle_{\Eho\, \cup \, \calF_h^D}.=(\bfe, \bv)_{\Th} +\langle \bg, \bv \rangle_{\calF_h^N}\quad \forall\, \bv \in \V.
\end{equation}
The extra stabilization term, which has no physical meaning, is required for ensuring (at the theoretical level) the stability of the methods. 
However as we will demonstrate via numerical experiments, the presence of such term might degrade the overall performance of numerical methods (see Section \ref{num_res}).

%%%%%%%%%%%%%%%%%%%%%%%%%%%%%%%%%%%%%%%%%%%
\section{Stability}\label{stability:sec}
The main goal of this section is to prove stability in the natural energy norm induced by the DG methods described in Section \ref{DG-methods}. We first introduce some notation and state the main stability results for both  displacement-stress and displacement formulations. After discussing these results, we carry out their proofs.\\

For the DG methods in  displacement-stress formulation \eqref{dg:mixform} we define the energy norm
\begin{equation}\label{eq:normMDG}
\normMDG{(\bu^h,\bsig^h)}^{2}= \|\rho^{1/2}\bu^{h}_t\|_{0,\Th}^{2} +\|\mathcal{A}^{1/2}\bsig^{h}\|_{0,\Th}^{2} + \|\cuu^{1/2}\salto{\bu^{h}}\|_{0,\Eho\cup \EhD}^{2} + \|\cdd^{1/2}\salto{\bsig^{h}}\|_{0,\Eho\cup \EhN}^{2}\;,  
\end{equation}
for all  $(\bu^{h},\bsig^h) \in \V\times \bS_h$. For the LDG and ALT methods, one needs to set above $\cdd= {\bf 0}$ and $\cuu= \cdd = {\bf 0}$, respectively. For the DG methods in displacement formulation, the energy norm is defined as
\begin{equation}\label{norm:ip}
\normIP{\bu^{h}}^{2}= \|\rho^{1/2} {\bu}^h_{t}\|_{0,\Th}^{2}  +\|\calD^{1/2} \beps(\bu^h)\|_{0,\Th}^{2}+ \| \SF^{1/2}\salto{ \bu^h}\|_{0,\Eho\cup \EhD}^{2} \quad \forall\, \bu^{h} \in \V.
\end{equation}
For further use, we also define the norm \begin{equation}\label{eq:normA}
\normA{\bu^{h}}^{2} = \|\mathcal{D}^{1/2}\beps{(\bu^{h})}\|_{0,\Th}^{2} + 
\sum_{F \in \Eho\, \cup \, \calF_h^D}\norm{\av{\calD}^{1/2}h_F^{-1/2}\salto{\bu^{h}}}_{0,F}^2.
\end{equation}
The main results of this section are contained in the following two propositions. Since the results and their proofs differ slightly for  displacement-stress and displacement formulations, we have chosen to state these results separately. 
\begin{proposition}\label{stab:mixed}
Let $(\bu^{h},\bsig^{h}) \in \V  \times \bS_h$ be the approximate solution obtained with any of the DG methods for the  displacement-stress formulation introduced in Section~\ref{sec:mixedDG}. 
 \begin{enumerate}
 \item In the absence of external forces, i.e.,  $\bff= \bg= {\bf 0}$,  FDG, LDG and ALT methods are fully conservative: 
\begin{equation*}
\begin{aligned}
\normMDG{(\bu^h(t),\bsig^h(t))} = \normMDG{(\bu^{h}_{0},\bsig^{h}_{0})}, 
&& 0<t\leq T.
\end{aligned}
\end{equation*}
\item If $\bff \in L^{2}((0,T]; \bL^{2}(\O))$ and $\partial \Omega=\Gamma_D$, the FDG, LDG and ALT methods satisfy the following  \emph{a priori} discrete energy estimate: 
\begin{equation*}
\begin{aligned}
& \normMDG{(\bu^h(t), \bsig^h(t))}
\lesssim \normMDG{(\bu^h_0, \bsig^h_0)}
+  T \int_0^t  \rinf^{-1/2} \| \bff(\tau) \|_{0,\O}d\tau
&& 0<t\leq T.
\end{aligned}
 \end{equation*}
 \item If $\bff \in L^{2}((0,T]; \bL^{2}(\O))$ and $\bg\in C^{1}((0,T]; {\bf H}^{1}(\Gamma_N))$, the FDG and LDG methods satisfy the following \emph{a priori} discrete energy estimate: for all $0<t\leq T$
 \begin{equation*}
 \normMDG{(\bu^h(t),\bsig^h(t))} 
\lesssim   \sqrt{ \mathcal{G}_{\textrm{MDG}}}
+ T  \int_0^t \left( 
\rinf^{-1/2} \| \bfe(\tau) \|_{0,\O}
+  \Dinf^{-1/2} \norm{\bg_{\tau}(\tau)}_{1,\Gamma_N} 
\right)  \, d\tau,
%+ \Dinf^{-1/2} \|\bg_{\tau}(\tau)\|_{1/2, \Gamma_N}
 \end{equation*} 
where $\bg_{\tau}$ denotes the time derivative of $\bg$ and
%\begin{multline*}
%\mathcal{G}_{\textrm{MDG}} =\normMDG{(\bu^h_0, \bsig^h_0)}^{2} 
 %+\Dinf^{-1} \sup_{0<t\leq T} \|\bg(t)\|^2_{1/2, \Gamma_N} 
%+ \Dinf^{-1}  \|\bg_{0}\|^2_{1/2, \Gamma_N}\\
%+\Dinf^{-1} \sup_{0<t\leq T}
%\norm{\bg(t)}_{1,\Gamma_N}^2 
%+\Dinf^{-1}\norm{\bg_0}_{1,\Gamma_N}^2.
 %\end{multline*}

 \begin{equation*}
\mathcal{G}_{\textrm{MDG}} =\normMDG{(\bu^h_0, \bsig^h_0)}^{2} 
+\Dinf^{-1} (\norm{\bg_0}_{1,\Gamma_N}^2+\sup_{0<t\leq T}
\norm{\bg(t)}_{1,\Gamma_N}^2 
 )+T \int_0^t \Dinf^{-1}  \norm{\bg_{\tau}}_{1,\Gamma_N} 
\|\bg\|_{1/2, \Gamma_N} \,d\tau \;.
 \end{equation*} 
  \end{enumerate}
\end{proposition}
For the IP($\delta$) method, the stability result reads as follows.
\begin{proposition}\label{stab:prim}
 Let $\bu^{h} \in \V$ be the  approximate solution obtained with the SIP($\delta$) method introduced in Section~\ref{sec:DGprimal}. Moreover, suppose that the penalty parameter $\czz$ appearing in the definition of the penalty function  \eqref{def:penalty_SIP} is chosen large enough. Then,
 \begin{enumerate}
 \item In the absence of external forces, i.e.,  $\bff= \bg= {\bf 0}$,
\begin{equation*}
\begin{aligned}
& \normIP{\bu^{h}(t)}
\lesssim \normIP{\bu^{h}_{0}},  
&& 0<t\leq T;
\end{aligned}
\end{equation*}
\item If $\bff \in L^{2}((0,T]; \bL^{2}(\O))$ and $\bg\in C^{1}((0,T]; {\bf H}^{1}(\Gamma_N))$, then
\begin{equation*}\label{stab:primal_II}
\begin{aligned}
&\normIP{\bu^{h}(t)}^{2} \lesssim 
\sqrt{\mathcal{G}_{\textrm{IP}}} 
+T \int_0^t  \left( \rinf^{-1} \| \bff(\tau) \|_{0,\O} +  \|\bg_{\tau}(\tau)\|_{1, \Gamma_N}\right) \, d\tau
&& 0<t\leq T,
\end{aligned}
\end{equation*}
where
\begin{equation*}
\mathcal{G}_{\textrm{IP}} =\normIP{\bu^{h}_{0}}^{2} 
+\Dinf^{-1} \sup_{0<t\leq T} \|\bg(t)\|_{1, \Gamma_N}^{2} 
+ \Dinf^{-1} \|\bg_{0}\|_{1, \Gamma_N}^{2}.
\end{equation*}
\end{enumerate}
\end{proposition}
We next discuss the stability results, putting them also in context with previous analysis.
\begin{enumerate}
\item In the case of boundary conditions of mixed type, we note that Proposition ~\ref{stab:mixed} and Proposition~\ref{stab:prim} require further regularity on the traction boundary data $\bg$  compared the one required for the continuous problem. Whether this is a technical restriction due to an artifact of our proof or really needed to ensure stability of the methods is not clear at the present time but will be subject of future research. The restriction comes into play from the proof of Lemma~\ref{le:aux00}, see below, although such proof try to mimic the corresponding one for the continuous problem.

\item  For displacement formulation, our analysis applies to symmetric IP; it does not cover the non-symmetric NIP ($\theta=1$) and IIP ($\theta=0$) methods, cf. Section~\ref{sec:DGprimal}, but also seem to indicate they are not the most natural option to discretize this problem. The nonsymmetry of the bilinear form $a(\cdot,\cdot)$ precludes for showing stability with the present analysis.
In \cite{RiShWhWh03,AntoniettiMazzieriQuarteroniRapetti_2012} the authors consider the variant of the IP  discretization described in \eqref{extraAA} with an  extra stabilization term that penalizes the time derivative of the displacement.  The  physical meaning of such extra penalization is unclear, but allows for carrying out the proof of stability. However, such artifact seems to degrade the overall performance of numerical methods (see Section \ref{num_res}).
\end{enumerate}

We next state two auxiliary results that will be required to prove Propositions \ref{stab:mixed} and \ref{stab:prim}: their proofs are given in Appendix~\ref{sec:ProofAuxiliary}. 

\begin{lemma}\label{le:aux00}
 Let $\bff \in L^{2}((0,T]; \bL^{2}(\O))$ and $\bg\in C^{1}((0,T]; {\bf H}^{1/2}(\Gamma_N))$. Let $(\bu^{h}, \bsig^h) \in \V \times \bS_h $ be the
DG approximation to the  solution $(\bu, \bsig)$ of problem \eqref{eqd:1}--\eqref{eqd:6} obtained with any of the DG methods introduced in Section~\ref{DG-methods}. Then, the following bounds hold:
 \begin{align}
\left| \int_0^t  ( \bfe(\tau) ,  \bu^{h}_\tau(\tau) )_{\Th} \,d\tau \right| &\leq t  \int_0^t  \rinf^{-1/2} \| \bfe(\tau) \|_{0,\O}
\|\rho^{1/2} \bu^{h}_\tau(\tau) \|_{0,\Th}  \, d\tau\;, \label{aux:00}&&\\
\left|  \int_0^t  \langle \cdd \bg_\tau(\tau), \bsig^{h}(\tau)\,\n\rangle_{\EhN} \,d\tau \right| &\lesssim t\int_0^t   \Dinf^{-1/2} \|\bg_\tau(\tau)\|_{1/2, \Gamma_N} \|\cdd^{1/2} \bsig^{h}(\tau)\n\|_{0,\EhN} d\tau, &&\label{aux:02}
\end{align}
where $\cdd$ is defined as in \eqref{las-ces} and
$\Dinf$, $\rinf$ are given in \eqref{cota:D} and \eqref{material:0}, respectively. Furthermore, if $\bg\in C^{1}((0,T]; {\bf H}^{1}(\Gamma_N))$, then for any $\epsilon >0$, it holds
\begin{multline}\label{aux:01}
  \abs{ \int_0^t  \langle \bg(\tau), \bu^{h}_{\tau}(\tau)\rangle_{0,\EhN} \,d\tau}   
\lesssim
\epsilon \normA{\bu^{h}(t)}^2 
+  \Dinf^{-1/2}  \normA{\bu^{h}_0}\norm{\bg_{0}}_{1,\Gamma_N}
+
\frac{ \Dinf^{-1}}{\epsilon}
\norm{\bg (t)}_{1,\Gamma_N}^2 \\
+ t\, \int_0^t  \Dinf^{-1/2} \|\bg_{\tau} (\tau)\|_{1,\Gamma_N}\normA{\bu^{h }(\tau)} \,d\tau.
\end{multline}
\end{lemma}
The following result provides a bound of the norm of the symmetric discrete gradient in terms of the discrete stress tensor, and will be required in proof of Proposition~\ref{stab:mixed}.
\begin{lemma}\label{le:aux04} 
 Let $\bff \in L^{2}((0,T]; \bL^{2}(\O))$ and $\bg\in C^{1}((0,T]; {\bf H}^{1/2}(\Gamma_N))$. Let $(\bu^{h},\bsig^{h}) \in \V  \times \bS_h$ be the approximate solution to \eqref{eqd:1}-\eqref{eqd:6} obtained with the FDG or the LDG methods introduced in Section~\ref{sec:mixedDG}. Then, the following bound holds:
\begin{equation}\label{eq:aux04}
\|\mathcal{D}^{1/2}\beps{(\bu^{h})}\|_{0,\Th}
\lesssim 
\|\mathcal{A}^{1/2}\bsig^{h} \|_{0,\Th}
+ 
\| \cuu^{1/2} \salto{\bu^h}\|_{0,\Eho}
+
 \| \cdd^{1/2} \salto{\bsig^h} \|_{0,\EhN}
+
  \Dinf^{-1}\| \bg \|_{1/2, \Gamma_N}^{2},
\end{equation}
where $\Dinf$ is the lower bound on the stiffness operator $\mathcal{D}$ as given in \eqref{cota:D}.
For the LDG method the last two terms on the right hand side are not present in the bound.
\end{lemma}

\subsection{Proof of Proposition~\ref{stab:mixed}}
\begin{proof}[Proof of Proposition~\ref{stab:mixed}]
We mimic the proof of the stability result for the continuous problem.\\
\noindent {\bf Step 1.}  We take $\bv= {\bu}^h_t \in \V$ as test function in the first equation of \eqref{dg:3a} and use that $\salto{\widehat{\bsig}}= \bf{0}$, to obtain
\begin{equation}\label{le2:1}
(\rho {\bu}^h_{tt}, {\bu}^h_t )_{\Th} +(\bsig^h , \beps( {\bu}^h_t ))_{\Th}   -\langle \av{\widehat{\bsig}}, \salto{{\bu}^h_t}\rangle_{\Eho\, \cup \,\EhD} =(\bfe, {\bu}^h_t)_{\Th} +\langle \bg, {\bu}^h_t\rangle_{ \calF_h^N}.
\end{equation}
\noindent {\bf Step 2.} We consider the DG approximation of  the constitutive equation \eqref{eqd:2} differentiated with respect to time
\begin{multline*}
( \calA {\bsig}_t^h , \btau)_{\Th} +( \beps({\bu}^h_t), \btau)_{\Th}+\langle \av{\widehat{\bu}_{t}-\bu^h_t }, \salto{\btau}\rangle_{\Eho}  +\langle \salto{\widehat{\bu}_{t}-{\bu}^h_t}, \av{\btau}\rangle_{\Eho\, \cup \,\EhD} \\ 
  +\langle \cdd (\bsig^{h}_{t}\bn -\bg_{t}), \btau\,\n\rangle_{\EhN}  =0 \end{multline*}
for all $\btau \in \bS_h$,  
where the numerical flux $\widehat{\bu}_t$ is defined according to the definition of $\widehat{\bu}$. In particular on boundary faces we have $\widehat{\bu}_t=\bf{0}$, mimicking the boundary condition for \eqref{2dt}, $\widehat{\bu}_t={\bf 0} $ on $\Gamma_D$.  By setting  $\btau =\bsig^h$ in the above equation, and using that $\salto{\widehat{\bu}_t}={\bf 0}$ we get,
\begin{multline}\label{le2:3}
( \calA {\bsig}^h_t , \bsig^h)_{\Th} -( \beps({\bu}^h_t), \bsig^h)_{\Th}-\langle \av{\widehat{\bu}_t - {\bu^h_t}}, \salto{\bsigh\,}\rangle_{\Eho}  
+\langle \salto{  {\bu}^h_t}, \av{\bsig^h}\rangle_{\Eho\, \cup \,\EhD}   \\
+\langle \cdd \bsig^{h}_{t},\bsig^{h}\, \n \rangle_{\EhN} 
= \langle \cdd \bg_{t}, \bsig^{h}\, \n \rangle_{\EhN}. 
\end{multline}
{\bf Step 3.}  Summing  up the equations \eqref{le2:1} and \eqref{le2:3}, we have
\begin{equation}\label{le2:4}
(\rho {\bu}^h_{tt}, \bu^h_t )_{\Th}+( \calA \bsig^h_t , \bsig^h)_{\Th}   +\mathcal{Q} 
=(\bfe, \bu^h_t)_{\Th} +\langle \bg, {\bu}^h_{t}\rangle_{\calF_h^N}+\langle \cdd \bg_{t}, \bsig^{h}\, \n \rangle_{\EhN}, 
\end{equation}
where $\mathcal{Q}$ is defined by
\begin{equation}\label{defQ:1}
\mathcal{Q} =-\langle \av{\widehat{\bsig}}, \salto{{\bu}^h_t}\rangle_{\Eho\, \cup \,\EhD}  + \langle \salto{\bu^h_t}, \av{\bsig^h}\rangle_{\Eho\, \cup \,\EhD}  -\langle \av{\widehat{\bu}_{t}-{\bu^h_t}}, \salto{\bsig^h\,}\rangle_{\Eho}  + \langle \cdd \bsig^{h}_{t},\bsig^{h}\rangle_{\EhN}.
\end{equation}
Equation \eqref{le2:4} is then equivalent to
\begin{equation}\label{le2:5}
\frac{1}{2} \frac{d}{dt} \left(\|\rho^{1/2} {\bu}_t^h\|^{2}_{0,\Th}+\| \calA^{1/2} \bsig^h\|^{2}_{0,\Th}  \right)+\mathcal{Q} =(\bfe, {\bu}_t^h)_{\Th} +\langle \bg, {\bu}_t^h\rangle_{ \calF_h^N}+\langle \cdd \bg_{t}, \bsig^{h}\, \n \rangle_{\EhN}.
\end{equation}
We first study the case \emph{(i)}, i.e., $\bff= \bg= \bf{0}$.  Then, we claim, that to guarantee stability of the method it is enough to show that $\mathcal{Q}$ is either non-negative or it can be rewritten as the time derivative of a non-negative quantity.
Substituting in \eqref{defQ:1} the definition of the fluxes \eqref{flux:dg} for the FDG methods, $\mathcal{Q}$ becomes  
\begin{multline*}
\mathcal{Q}^{FDG}
=-\langle  \av{\bsig^{h}}_{\delta} - \av{\bsig^{h}}, \salto{{\bu}^h_t}\rangle_{\Eho\, \cup \,\EhD}  
+  \langle \cuu \salto{\bu^{h}}, \salto{{\bu}^h_t}\rangle_{\Eho\cup \EhD}   \\
+\langle \av{{\bu^h_t}} - \av{{\bu^h_t}}_{1-\delta}, \salto{\bsig^h\,}\rangle_{\Eho}
+\langle \cdd\salto{\bsig_t^h\,}, \salto{\bsig^h\,}\rangle_{\Eho}
+ \langle \cdd \bsig^{h}_{t},\bsig^{h}\rangle_{\EhN}.
\end{multline*}
Thanks to the definition of the average operator on boundary edges/faces and the identity \eqref{eq:media-mediad}, we have
\begin{equation}\label{relation00}
\langle  \av{\bsig^{h}}_{\delta}-\av{\bsig^{h}}, \salto{{\bu}^h_t}\rangle_{\Eho\, \cup \,\EhD} =
\langle  \av{{\bu^h_t}}- \av{{\bu^h_t}}_{1-\delta},\salto{\bsig^h}\rangle_{\Eho},
\end{equation}
and therefore
\begin{equation}\label{ecco:1}
\mathcal{Q}^{FDG}=\frac{1}{2}\frac{d}{dt} \left( \|\cuu^{1/2}  \salto{{\bu}^h}\|_{0,\Eho\cup \EhD}^{2} + \|\cdd^{1/2}  \salto{\bsig^h}\|_{0,\Eho\cup \EhN}^{2}\right).
\end{equation}
For the LDG ($\cdd={\bf 0}$) and the ALT ($\cuu=\cdd={\bf 0}$) methods the above expression reduces to
\begin{align*}
& \mathcal{Q}^{LDG}=\frac{1}{2}\frac{d}{dt}  \|\cuu^{1/2}  \salto{{\bu}^h}\|_{0,\Eho\cup \EhD}^{2},
&& \mathcal{Q}^{ALT}=0.
\end{align*}
Therefore, for all the considered methods, the corresponding discrete energy as defined in \eqref{eq:normMDG} is preserved in time, which proves part \emph{(i)} of the thesis.\\

Next we deal with the cases $\bff\ne {\bf 0}$, $\Gamma_D=\partial \Omega$ and $\bff,  \bg \ne {\bf 0} $. 
For the former case, by using estimate \eqref{aux:00} from Lemma~\ref{le:aux00}, we  find
\begin{equation*}
\normMDG{(\bu^h, \bsig^h)}^{2}  
 \lesssim  \normMDG{( \bu^{h}_{0}, \bsig^{h}_{0} )}^{2} 
+  2t \, \int_0^t  \rinf^{-1/2} \| \bff \|_{0,\O} \| \rho^{1/2} \bu^h_\tau \|_{0,\Th}d\tau,
\end{equation*}
which together with the definition \eqref{eq:normMDG}, $t\leq T$ and a standard application of Gronwall's Lemma~(cf. Lemma~ \ref{Gronwall}) gives the result and proves part \emph{(ii)}.

We finally show  part \emph{(iii)}.
We restrict ourselves to the FDG formulation;
the corresponding estimate for the LDG can be obtained by setting $\cdd={\bf 0}$. Substituting \eqref{ecco:1} into \eqref{le2:5} gives
\begin{multline}\label{per-sust}
\frac{1}{2} \frac{d}{dt} \left(\|\rho^{1/2} \bu_t^h\|^{2}_{0,\Th}+\| \calA^{1/2} \bsig^h\|^{2}_{0,\Th}+ \|\cuu^{1/2}\salto{\bu^h}\|^{2}_{0,\Eho\cup\EhD} +\|\cdd^{1/2}\salto{\bsig^h}\|^{2}_{0,\Eho\cup\EhN}\right)\\
=(\bfe, \bu_t^h)_{\Th} +\langle \bg, \bu_t^h\rangle_{ \EhN}+\langle \cdd \bg_{t}, \bsig^{h}\, \n \rangle_{\EhN}.
\end{multline}
Recalling now the definition of the $\normMDG{\cdot}$-norm \eqref{eq:normMDG}, and integrating in time we get 
\begin{multline*}\label{magari00}
\frac{1}{2}\normMDG{(\bu^h, \bsig^h)(t)}^{2} 
\leq
\frac{1}{2}\normMDG{(\bu^{h}_{0}, \bsig^{h}_{0})}^{2} 
+ \underbrace{\left|\int_0^t  ( \bfe ,  \bu^{h}_\tau )_{\Th} \, d\tau \right|}_{\text{I}}
+ \underbrace{\left| \int_0^t \langle \bg , \bu^{h}_{\tau} \rangle_{\EhN}\,d\tau\right|}_{\text{II}}\\
+\underbrace{\left|\int_0^t  \langle \cdd \bg_\tau, \bsig^{h}\,\n\rangle_{\EhN} \, d\tau \right|}_{\text{III}}\;,
\end{multline*}
where the last term is only present for the FDG method.  
The terms I and III are readily estimated by using Lemma~\ref{le:aux00} 
\begin{equation*}
\begin{aligned}
&\text{I} \leq t \int_0^t  \rinf^{-1/2} \| \bfe \|_{0,\O}\|\rho^{1/2} \bu^{h}_\tau \|_{0,\Th}  \, d\tau,  &&\text{III}\lesssim t \int_0^t  \Dinf^{-1/2} \|\bg_{\tau}\|_{1/2, \Gamma_N}\|\cdd^{1/2} \bsig^{h}\n\|_{0,\EhN} \,d\tau\;. 
\end{aligned}
\end{equation*}
To estimate the term II,  from Lemma~\ref{le:aux00}  we first have
\begin{equation*}
\begin{aligned}
\text{II}&\lesssim
\epsilon \normA{\bu^{h}}^2 +  \normA{\bu^{h}_0}\norm{\bg_{0}}_{1,\Gamma_N}
+\frac{ \Dinf^{-1}}{\epsilon}
\norm{\bg}_{1,\Gamma_N}^2 
+t \int_0^t \Dinf^{-1/2} \norm{\bg_{\tau}}_{1,\Gamma_N}\normA{\bu^{h}} \,d\tau,
\end{aligned}
\end{equation*}
with $\epsilon>0$ (to be specified later on). Now, to bound  $\normA{\bu^{h}} $ in terms of the $\normMDG{(\bu^h,\bsig^h)}$ norm,
we use estimate  
\eqref{eq:aux04} from Lemma~\ref{le:aux04},  to get
\begin{equation*}
\normA{\bu^{h}}^{2} 
\leq C_{\text{II}}(
\normMDG{(\bu^h,\bsig^h)}^{2}
+ \Dinf^{-1} \|\bg\|^2_{1/2, \Gamma_N} ),
\end{equation*}
and so the estimate for II becomes,
\begin{align*}
\text{II}&\lesssim
\epsilon  C_{\text{II}}
\normMDG{(\bu^h,\bsig^h)}^{2} + \normA{\bu^{h}_0}\norm{\bg_{0}}_{1,\Gamma_N}  +\frac{ \Dinf^{-1}}{\epsilon}
\norm{\bg}_{1,\Gamma_N}^2 +  \epsilon C_{\text{II}}\Dinf^{-1} \|\bg\|^2_{1/2, \Gamma_N}\\
&+
t \int_0^t \Dinf^{-1/2} \norm{\bg_\tau}_{1,\Gamma_N}
\normMDG{(\bu^h,\bsig^h)}\,d\tau
+t \int_0^t \Dinf^{-1}  \norm{\bg_{\tau}}_{1,\Gamma_N} 
\|\bg\|_{1/2, \Gamma_N} \,d\tau
\end{align*}
Substituting all the above estimates, recalling the definition of the $\normMDG{\cdot}$ norm, using standard Sobolev' imbeddings  and taking $\epsilon$ so that $1/2-C_{\text{II}}\epsilon$ is positive,  gives
\begin{multline*}
\normMDG{(\bu^h, \bsig^h)}^{2} 
\lesssim \normMDG{(\bu^{h}_{0}, \bsig^{h}_{0})}^{2} + 
\Dinf^{-1/2}\normA{\bu^{h}_0}\norm{\bg_{0}}_{1,\Gamma_N}  + \Dinf^{-1}
\norm{\bg}_{1,\Gamma_N}^2 
\\
 +t \int_0^t \Dinf^{-1}  \norm{\bg_{\tau}}_{1,\Gamma_N} 
\|\bg\|_{1/2, \Gamma_N} \,d\tau +t \int_0^t \left( 
\rinf^{-1/2} \| \bfe \|_{0,\O}
+  \Dinf^{-1/2} \|\bg_{\tau}\|_{1,\Gamma_N} 
\right)\normMDG{(\bu^{h}, \bsig^h)}  \, d\tau.
\end{multline*}
Finally, using $t\leq T$ and a standard application of Gronwall's Lemma~gives the estimate in part \emph{(iii)} and concludes the proof.
 \end{proof}
%%%%%
%%%%%
%%%%%
\subsection{Proof of Proposition~\ref{stab:prim}}
We first observe that, for any $F\in \Eho\, \cup \, \EhD$, and any $\bw, \bv \in \V$, the Cauchy-Schwarz, Agmon \eqref{agmon} and inverse \eqref{eq:inverse} inequalities give 
\begin{equation}\label{coso:5}
\begin{aligned}
\left| \langle  \av{\calD \beps(\bw)}_{\delta}, \salto{\bv} \rangle_{F} \right| 
&\lesssim \frac{1}{\czz} 
\|\calD^{1/2} \beps(\bw)\|_{0,\K}\|\SF^{1/2}   \salto{\bv}\|_{0,F}^{2}\leq 
\frac{1}{\czz} 
\normIP{\bw}\normIP{\bv}
\end{aligned}
\end{equation} 
where $\czz$ is the positive parameter appearing in the definition of the penalty function \eqref{def:penalty_SIP}.
\begin{proof}[Proof of Proposition~\ref{stab:prim}]
The proof follows the same ideas of the proof of stability in the continuous case, but now we work directly with the displacement variational formulation.\\
{\bf Step 1.}  We  set $\bv= {\bu}^h_t \in \V$ and $\theta=-1$  (SIPG($\delta$)-method) in \eqref{dg:ip3} to get
%\begin{align*}
%\frac{1}{2} \frac{d}{dt} \Big(  
% \normIP{\bu^{h}}^{2} %\right.
%&\left. 
%-2 \langle \av{\calD \beps(\bu^h)}_{\delta}, \salto{{\bu}^h}\rangle_{\Eho\, \cup \, \EhD} \right) \\
%&=
%\begin{multlined}[t]
%-(1+\theta) \langle \salto{\bu^h}, \av{\calD \beps({\bu}^h_t)}_{\delta} \rangle_{\Eho\, \cup \,\EhD}
%+(\bfe, {\bu}^h_t)_{\Th} +\langle \bg, {\bu}^h_t \rangle_{\calF_h^N},
%\end{multlined}
%\\
%&=
%\begin{multlined}[t]
% -(1+\theta) \frac{d}{dt} \langle \salto{\bu^h}, \av{\calD \beps({\bu}^h)}_{\delta} \rangle_{\Eho\, \cup \,\EhD} 
%+ (1+\theta)  \langle \salto{\bu^h_t}, \av{\calD \beps({\bu}^h)}_{\delta} \rangle_{\Eho\, \cup \,\EhD} \\
%+(\bfe, {\bu}^h_t)_{\Th} 
%+\langle \bg, {\bu}^h_t \rangle_{\calF_h^N}.
%\end{multlined}
%\end{align*}
%Adding now on both sides the term $(1+\theta) \frac{d}{dt} \langle \salto{\bu^h}, \av{\calD \beps({\bu}^h)}_{\delta} \rangle_{\Eho\, \cup \,\EhD}$ gives
%\begin{multline}\label{eq:stab_SIPG}
%\frac{1}{2} \frac{d}{dt} 
%\left(  
% \normIP{\bu^{h}}^{2} + 2 \theta \langle \av{\calD \beps(\bu^h)}_{\delta}, \salto{{\bu}^h}\rangle_{\Eho\, \cup \, \EhD} \right)=\\
%(1+\theta)  \langle \salto{\bu^h_t}, \av{\calD \beps({\bu}^h)}_{\delta} \rangle_{\Eho\, \cup \,\EhD} 
%+(\bfe, {\bu}^h_t)_{\Th} 
%+\langle \bg, {\bu}^h_t \rangle_{\calF_h^N}.
%\end{multline}
%%%%%%%%%%%%%%%%%%%%%%%%%%%%%%%%%%%%%%%%%%%%%%%
\begin{equation}\label{eq:stab_SIPG}
\frac{1}{2} \frac{d}{dt} \Big(  
 \normIP{\bu^{h}}^{2} %\right.
\left. 
-2 \langle \av{\calD \beps(\bu^h)}_{\delta}, \salto{{\bu}^h}\rangle_{\Eho\, \cup \, \EhD} \right) 
=
(\bfe, {\bu}^h_t)_{\Th} +\langle \bg, {\bu}^h_t \rangle_{\calF_h^N}.
\end{equation}
\noindent {\bf Step 2.} 
Integrating in time the above equation we obtain 
\begin{multline}\label{eq:wsip}
\normIP{\bu^{h}}^{2} 
-  2  \langle \av{\calD \beps(\bu^{h})}_{\delta}, \salto{{\bu}^h}\rangle_{\Eho\, \cup \, \EhD}
=
\normIP{\bu^{h}_{0}}^{2}\\
-2  \langle \av{\calD \beps(\bu^{h}_{0})}_{\delta}, \salto{{\bu}^h_0}\rangle_{\Eho\, \cup \, \EhD}
+ 2 \int_{0}^{t}  (\bfe, {\bu}^h_\tau)_{\Th} \, d\tau 
+ 2 \int_{0}^{t} \langle \bg, {\bu}^h_\tau \rangle_{\calF_h^N} 	\, d\tau.
\end{multline}
To guarantee stability we first need to show that the sum of the terms on the right hand side is an upper bound for $\normIP{\bu^{h}}^{2}$.  Using \eqref{coso:5}, the arithmetic-geometric inequality and choosing the penalty parameter $\czz$  sufficiently large, we obtain
\begin{align*}
\normIP{\bu^{h}}^{2} -2 \langle \av{\calD \beps(\bu^{h})}_{\delta}, \salto{{\bu}^h}\rangle_{\Eho\, \cup \, \EhD}  
\gtrsim \normIP{\bu^{h}}^{2}.
 \end{align*}
Using \eqref{coso:5} we also obtain
\begin{equation*}
 \normIP{\bu^{h}_{0}}^{2} 
 -2 \langle \av{\calD \beps(\bu^{h}_{0})}_{\delta}, \salto{{\bu}^h(0)}\rangle_{\Eho\, \cup \, \EhD} 
\lesssim  \normIP{\bu^{h}_{0}}^{2}.
 \end{equation*}
Substitution now of these two estimates into \eqref{eq:wsip}, gives
\begin{equation*}
 \normIP{\bu^{h}}^{2} 
\lesssim
\normIP{\bu^{h}_{0}}^{2} 
+   \int_{0}^{t}  (\bfe, {\bu}^h_\tau)_{\Th} \, d\tau 
+  \int_{0}^{t} \langle \bg, {\bu}^h_\tau \rangle_{\calF_h^N} 	\, d\tau.
\end{equation*}
Then, in case of no external forces, i.e., $\bff= \bg={\bf 0}$ (the last two integrals above would vanish) the above estimate yield to part \emph{(i)} of the thesis.
As regards part {\it (ii)},  Lemma~\ref{le:aux00}  and the inequality  
$\normA{\bu^{h}} \leq \normIP{\bu^{h}}$, give for $\epsilon>0$
\begin{multline*}
\normIP{\bu^{h}}^{2} \lesssim \epsilon  \normIP{\bu^{h}}^{2} + 
\normIP{\bu^{h}_{0}}^{2} +\normA{\bu^{h}_0}\norm{\bg_{0}}_{1,\Gamma_N} +\frac{\Dinf^{-1}}{\epsilon}\|\bg\|_{1, \Gamma_N}^{2}\\
+ t \int_0^t  \left( \rinf^{-1} \| \bff \|_{0,\O} + \Dinf^{-1/2} \|\bg_{\tau}\|_{1, \Gamma_N}\right)
\normIP{\bu^{h}} \,d\tau.
 \end{multline*}
Therefore  choosing $\epsilon$ small enough and $t\leq T$ we obtain
\begin{multline*}
\normIP{\bu^{h}}^{2} \lesssim 
\normIP{\bu^{h}_{0}}^{2} +\normA{\bu^{h}_0}\norm{\bg_{0}}_{1,\Gamma_N} 
+\Dinf^{-1} \|\bg\|_{1, \Gamma_N}^{2} \\
+T\int_0^t  \left( \rinf^{-1} \| \bff \|_{0,\O} +  \Dinf^{-1/2}\|\bg_{\tau}\|_{1, \Gamma_N}\right) \normIP{\bu^{h}} \, d\tau,
\end{multline*}
and the proof is complete after a standard application of Gronwall's Lemma~\ref{Gronwall}.
\end{proof}
%%%%%%%%%%%%%%%%%%%%%%%%%%%%%%%%%%%%%%%%%%%%%%%%%%%%%%%%%%%%%%%%%%%%%%%
%%%%%%%%%%%%%%%%%%%%%%%%%%%%%%%%%%%%%%%%%%%%%%%
\section{Error Analysis}
\label{sec:error_estimates}
In this section we derive  \emph{a priori} error estimates for the DG methods introduced in  Section~\ref{DG-methods}. To this aim, we introduce the following augment norms
\begin{equation*}
\begin{aligned}
\normtreMDG{(\bv,\bsig)}^2 &=\normMDG{(\bv,\bsig)}^2 
+
\norm{\cdd^{1/2} \av{\bsig}_{\delta}}_{0,\Eho\cup \EhD}^{2}
&& \forall\, (\bv,\bsig) \in {\bf H}^{1}(\Th)\times \vect{\mathcal{H}}^{1}(\Th),\\
\normtre{ \bv}_{\calE,\textrm{IP}}^2&= \norm{\bv}_{\calE,\textrm{IP}}^{2} + 
 \norm{h_F^{1/2}\av{\calD \beps(\bv)}_\delta}^2_{0,\EhoD}
&& \forall \,\bv \in {\bf H}^2(\Th), 
\end{aligned}
\end{equation*}
where $\normMDG{(\cdot,\cdot)}$ and $\norm{\cdot}_{\calE,\textrm{IP}}$ 
are defined in \eqref{eq:normMDG} and \eqref{norm:ip}, respectively. \\

For any $\bu \in {\bf H}^{k+1}(K)$ (resp. $\bsig \in \vect{\mathcal{H}}^{k+1}(K)$), $k\geq 0$, let $\bu^I$ (resp. $\bsig^I$) be the Lagrange nodal interpolants of $\bu$ (resp. $\bsig$). Then, using standard interpolation estimates, it holds
\begin{subequations}
\begin{align}
\normtreMDG{(\bu-\bu^I, \bsig-\bsig^I)}&\lesssim
h^{k}\left( |\bu |_{k+1,\O}^2 + h^2|\bsig |_{k+1,\O}^2  + h^2| \bu_t|_{k+1,\O}^2\right)^{1/2},
\label{eq:errMDGinterp}\\
\normtre{\bu-\bu^I}_{\calE,\textrm{IP}}
&\lesssim
h^{k}\left( |\bu |_{k+1,\O}^2 + h^2| \bu_t|_{k+1,\O}^2\right)^{1/2},
\label{eq:errIPinterp}\\
\|\calD^{1/2}\beps(\bu-\bu^I)\|_{0,\Th}
&\lesssim
h^{k} |\bu |_{k+1,\O},
\label{eq:errMDGinterp2}
\end{align}
\end{subequations}
where the hidden constants depend on the polynomial degree $k$, $\Dinf$, $\Dsup$, $\rsup$, and the shape regularity constant of the mesh $\Th$.
 \subsection{DG methods for the  displacement-stress formulations}
We now present the error analysis of the  displacement-stress DG methods. To derive the error equations, we first need the variational formulation for the continuous solution $(\bu,\bsig)$ of \eqref{eqd:1}--\eqref{eqd:6} allowing for test functions in $(\V,\bS_h)$, that reads as follows: for all $t \in (0,T]$ find $(\bu,\bsig) \in {\bf H}^{1}_{0,\Gamma_D}(\O)\times \vect{\mathcal L}^{2}(\O)$ such that
\begin{subequations}
\begin{align}
&(\rho {\bu}_{tt}, \bv )_{\Th} +(\bsig, \beps(\bv))_{\Th}   -\langle \av{\bsig}_\delta, \salto{\bv}\rangle_{\Eho\, \cup \,\EhD} =(\bfe, {\bv})_{\Th} +\langle \bg, {\bv} \rangle_{ \calF_h^N}, \quad&& \forall\, \bv \in \V\;, \label{misti_equ}\\
&( \calA \bsig_ , \btau)_{\Th} -( \beps(\bu), \btau)_{\Th}-\langle  \av{\bu}_{(1-\delta)}, \salto{\btau}\rangle_{\Eho} + \langle \av{\bu}, \salto{\btau}\rangle_{\Eho} =0,  \quad&& \forall\, \btau \in \bS_h,
\label{misti_eqs}
\end{align}
\end{subequations}
where to obtain last equation we have used that  $\salto{\bu} ={\bf 0}$ on $\Eho$ , so that
\begin{equation*}
0=\frac{2\delta-1}{2}\langle  \btau^{+} - \btau^{-} ,\salto{\bu} \rangle_{\Eho} = \frac{2\delta-1}{2}\langle  \bu^{+} - \bu^{-}, \salto{\btau} \rangle_{\Eho} =
-\langle  \av{\bu}_{(1-\delta)}, \salto{\btau}\rangle_{\Eho} + \langle \av{\bu}, \salto{\btau}\rangle_{\Eho}\;.
\end{equation*}
We define the error:
\begin{equation*}
(\eu,\es)=(\bu-\bu^{h},\bsig-\bsig^{h})\;.
\end{equation*}
Subtracting from \eqref{misti_equ}-\eqref{misti_eqs} the corresponding equations of the DG formulation in \eqref{dg:mixform}, and using that $\salto{\bsig} = \salto{\bu} ={\bf 0}$ on $\Eho$ together with the boundary conditions for the continuous solution \eqref{eqd:3}-\eqref{eqd:4} ($\bu = {\bf 0}$ on $\EhD$ and $\bsig\n=\bg$ on $\EhN$), we obtain the error equations:
\begin{subequations}
\begin{align}
&(\rho (\eut), \bv )_{\Th} 
+(\es , \beps( \bv ))_{\Th}   
-\langle \av{\es}_{\delta}, \salto{\bv}\rangle_{\Eho\cup \EhD} 
+\langle \cuu \salto{\eu}, \salto{\bv}\rangle_{\Eho\cup \EhD}=0,  &&\label{err:00}\\
&( \calA (\es) , \btau)_{\Th} -( \beps(\eu), \btau)_{\Th}+\langle \cdd\salto{\es}, \salto{\btau}\rangle_{\Eho\cup \EhN}  + \langle \salto{\eu}, \av{\btau}_{\delta}\rangle_{\Eho\cup \EhD} = 0, &&\label{err:01} 
\end{align}
\end{subequations}
where in last equation we have used  \eqref{relation00}.\\

The error analysis will mimic the stability proof, therefore we will also need to consider the variational formulation of the time derivative of equation \eqref{eqd:2} allowing $\btau\in \bS_h$.
B using the identity \eqref{eq:media-mediad}, we obtain
\begin{equation}\label{eq:errorsigmat}
\begin{aligned}
 ( \calA \bsig_t , \btau)_{\Th} -( \beps(\bu_t), \btau)_{\Th} 
 + \langle \cdd \salto{\bsig_t}, \salto{\btau}\rangle_{\Eho\cup \EhN} + \langle \salto{\bu_t}, \av{\btau}_{\delta} \rangle_{\Eho\cup \EhD} = \langle \cdd {\bf g}_t, \btau \bn \rangle_{\EhN}
\end{aligned}
\end{equation}
for all $\btau \in \bS_h$.
We now state the first main result of this section.
\begin{theorem}\label{Teo:ErrLDG}
Let $(\bu,\bsig)$ be the solution of \eqref{eqd:1}--\eqref{eqd:6}, and let $(\bu^h, \bsig^h) \in \bV_h\times \bS_h$ be the solution of any of the DG method in  displacement-stress formulations  defined in Section~\ref{sec:mixedDG}. 
Then,
\begin{multline}\label{eq:errorLDG}
\sup_{0<t \leq T}\normtreMDG{(\bu(t)-\bu^h(t), \bsig(t)-\bsig^h(t))}  \lesssim \\
h^{k}\sup_{0<t \leq T} \left( |\bu(t) |_{k+1,\O}^2 + h^2|\bsig(t) |_{k+1,\O}^2  + h^2| \bu_{\tau}(t)|_{k+1,\O}\right)^{1/2}\\
\qquad \qquad +h^{k}\, \int_0^T \left( |\bu(\tau) |_{k+1,\O}^2 + h^2|\bsig(\tau) |_{k+1,\O}^2  + h^2| \bu_{\tau}(\tau)|_{k+1,\O} \right)^{1/2} \, d\tau 
 \\
+h^{k}\, \int_0^T \left( |\bu_{\tau}(\tau) |_{k+1,\O}^2 + h^2|\bsig_{\tau}(\tau) |_{k+1,\O}^2  + h^2| \bu_{\tau\tau}(\tau)|_{k+1,\O} \right)^{1/2} \, d\tau.
\end{multline}  
where the hidden constant depends on $\Dinf$, $\Dsup$, $\rho^*$, the polynomial degree $k$,  the shape regularity constant of the mesh $\Th$ and has linear dependence on the observation time $T$.
\end{theorem}
%%%%
\begin{proof} 
The proof follows the proof of the stability result given in Proposition~\ref{stab:mixed}. 
We start by considering the splitting  $\eu = \bomegaI - \bomegah$
and $\es  = \bcsi^I - \bcsi^h$ with
\begin{equation}\label{splito}
\begin{aligned}
\bomegaI&= \bu -\bu^I,
&&\bomegah = \bu^h - \bu^I,
&& \bcsi^I= \bsig -\bsig^I,
&& \bcsi^h = \bsig^h - \bsig^I.
\end{aligned}
\end{equation}
We set $\bv = \bomegah_t \in \V$ in the error equation  \eqref{err:00} obtaining
\begin{equation}\label{err:1:1}
(\rho (\eut), \bomegah_t )_{\Th} 
+(\es , \beps( \bomegah_t ))_{\Th}   
-\langle \av{\es}_{\delta}, \salto{\bomegah_t}\rangle_{\Eho\cup \EhD} 
+\langle \cuu \salto{\eu}, \salto{\bomegah_t}\rangle_{\Eho\cup \EhD}=0, 
\end{equation}

For the second error equation, we 
reason as in the proof of Proposition~\ref{stab:mixed} (step II) and subtract from \eqref{eq:errorsigmat} its corresponding discrete approximation and take $\btau=\bcsi^h$
\begin{equation}\label{err:1:2}
\begin{aligned}
& ( \calA (\est) , \bcsi^h)_{\Th} -( \beps(\eup), \bcsi^h)_{\Th}+\langle \cdd\salto{\est}, \salto{\bcsi^h}\rangle_{\Eho\cup \EhN}  + \langle \salto{\eup}, \av{\bcsi^h}_{\delta}\rangle_{\Eho\cup \EhD}  =0\;.
 \end{aligned}
\end{equation}
Summing up equation \eqref{err:1:1} and \eqref{err:1:2} and using the splitting \eqref{splito} we  obtain
 \begin{multline*}
 (\rho (\bomegah_{tt}), \bomegah_t )_{\Th} + ( \calA (\bcsi^h_t) , \bcsi^h)_{\Th} +\langle \cuu \salto{\bomegah}, \salto{\bomegah_t}\rangle_{\Eho\cup \EhD} \\
+\langle \cdd\salto{\bcsi^h_t}, \salto{\bcsi^h}\rangle_{\Eho\cup \EhN} 
=
(\rho (\bomegaI_{tt}), \bomegah_t )_{\Th} + ( \calA (\bcsi^I_t) , \bcsi^h)_{\Th} \\
+\langle \cdd\salto{\bcsi^I_t}, \salto{\bcsi^h}\rangle_{\Eho\cup \EhN}
 + \langle \salto{\bomegaI_{t}}, \av{\bcsi^h}_{\delta}\rangle_{\Eho\cup \EhD}  -( \beps(\bomegaI_{t}), \bcsi^h)_{\Th} 
\\
+(\bcsi^I , \beps( \bomegah_t ))_{\Th}   
-\langle \av{\bcsi^I}_{\delta}, \salto{\bomegah_t}\rangle_{\Eho\cup \EhD} 
+\langle \cuu \salto{\bomegaI}, \salto{\bomegah_t}\rangle_{\Eho\cup \EhD}
 \end{multline*}
Notice that  the terms on the left hand side are just the time derivative of the $\normMDG{(\bomegah, \bcsi^h)}^{2}$, and therefore
 \begin{multline*}
 \frac{1}{2} \frac{d}{dt}   \normMDG{(\bomegah,\bcsi^h)}^{2}=
(\rho (\bomegaI_{tt}), \bomegah_t )_{\Th} 
+ ( \calA (\bcsi^I_t) , \bcsi^h)_{\Th}
+\langle \cdd\salto{\bcsi^I_t}, \salto{\bcsi^h}\rangle_{\Eho\cup \EhN} -( \beps(\bomegaI_{t}), \bcsi^h)_{\Th}\\
+\langle \cuu \salto{\bomegaI}, \salto{\bomegah_t}\rangle_{\Eho\cup \EhD}
-\langle \av{\bcsi^I}_{\delta}, \salto{\bomegah_t}\rangle_{\Eho\cup \EhD} + \langle \salto{\bomegaI_{t}}, \av{\bcsi^h}_{\delta}\rangle_{\Eho\cup \EhD} 
   +(\bcsi^I , \beps( \bomegah_t ))_{\Th}.
 \end{multline*}
Integration in time together with $\bomegah(0) = \bcsi^h(0) = {\bf 0}$, and the norm equivalence in $\V\times \bS_h$ 
\begin{equation*}
\normtreMDG{(\bomegah,\bcsi^h)}^2
\lesssim
\normMDG{(\bomegah,\bcsi^h)}^2  \lesssim |T_1| + |T_2| + |T_3|  + |T_4| + |T_5|,
\end{equation*}
where
\begin{eqnarray*}
T_1  & = & \int_0^t \left\{  (\rho (\bomegaI_{\tau\tau}), \bomegah_{\tau} )_{\Th} + ( \calA (\bcsi^I_{\tau}) , \bcsi^h)_{\Th} +\langle \cdd\salto{\bcsi^I_{\tau}}, \salto{\bcsi^h}\rangle_{\Eho\cup \EhN} -( \beps(\bomegaI_{\tau}), \bcsi^h)_{\Th}\right\}\,d\tau,\\
T_2 & = & \int_0^t \langle \cuu \salto{\bomegaI}\,d\tau, \salto{\bomegah_{\tau}}\rangle_{\Eho \cup \EhD}\,d\tau  =  \langle \cuu \salto{\bomegaI}, \salto{\bomegah}\rangle_{\Eho \cup \EhD} - \int_0^t \langle \cuu \salto{\bomegaI_{\tau}}, \salto{\bomegah}\rangle_{\Eho \cup \EhD}\,d\tau, \\
 T_3 & = & - \int_0^t\langle \av{\bcsi^I}_{\delta}, \salto{\bomegah_{\tau}}\rangle_{\Eho \cup \EhD} \,d\tau=  
-\langle \av{\bcsi^I}_{\delta}, \salto{\bomegah}\rangle_{\Eho \cup \EhD} + \int_0^t\langle \av{\bcsi^I_{\tau}}_{\delta}, \salto{\bomegah}\rangle_{\Eho \cup \EhD}\,d\tau, \\
 T_4 & = & - \int_0^t \langle \salto{\bomegaI_{\tau}}, \av{\bcsi^h}_\delta\rangle_{\Eho\cup \EhD}\,d\tau, \\
 T_5 & = &  \int_0^t (\bcsi^I , \beps( \bomegah_{\tau} ))_{\Th}\,d\tau =  (\bcsi^I , \beps( \bomegah ))_{\Th}   - \int_0^t (\bcsi^I_{\tau} , \beps( \bomegah ))_{\Th}\,d\tau, 
\end{eqnarray*}
where for $T_2$ ,$T_3$ and $T_5$ we have also employed the  integration by parts formula \eqref{eq:ibp_formula}.
To estimate $T_1$ we use Jensen and Cauchy-Schwarz  inequalities together with the definition of the norm $\normMDG{\cdot}$
\begin{equation*}
|T_1| \lesssim
t\int_0^t \left(\normtreMDG{(\bomegaI_{\tau},\bcsi^I_{\tau})} +\|\mathcal{D}^{1/2}\beps(\bomegaI_{\tau})\|_{0,\Th} \right) \normtreMDG{(\bomegah,\bcsi^h)} \,d\tau.
\end{equation*}
The estimate of $T_4$ follows easily from Jensen and Cauchy-Schwarz inequalities, the definition of $\cuu$ in \eqref{las-ces}, Agmon's inequality, and inverse inequality
\begin{eqnarray*}
|T_4 | \lesssim t\int_0^t
\norm{\cuu^{1/2} \salto{\bomegaI_{\tau}}}_{0,\Eho \cup \EhD} \norm{\cuu^{-1/2} \av{\bcsi^h}_\delta}_{0,\Eho \cup \EhD}\,d\tau
  \lesssim t \int_0^t
\norm{\cuu^{1/2} \salto{\bomegaI_{\tau}}}_{0,\Eho \cup \EhD} \|\calA^{1/2} \bcsi^h\|_{0,\Th}\,d\tau\;.
\end{eqnarray*}
Next, we observe that $T_2$ ,$T_3$ and $T_5$ can be estimated using  Jensen and Cauchy-Schwarz  inequalities together with the arithmetic-geometric inequality with some $\epsilon>0$ to be chosen later on. For $T_5$ and $T_2$ one immediately gets,
\begin{align*}
|T_5| & \lesssim  
\epsilon \norm{\calD^{1/2} \beps(\bomegah)}_{0,\Th}^2
+ \frac{1}{\epsilon} \norm{\calA^{1/2} \bcsi^I}_{0,\Th}^2 
+t \int_0^t  \norm{\calA^{1/2} \bcsi_{\tau}^I}_{0,\Th}   \norm{\calD^{1/2}\beps(\bomegah)}_{0,\Th}\,d\tau\;, \\
|T_2| & \lesssim
 \epsilon \normsmall{\cuu^{1/2}\salto{\bomegah} }^2_{0,\Eho \cup \EhD} + \frac{1}{\epsilon} 
\normsmall{ \cuu^{1/2}\salto{\bomegaI} }^2_{0,\Eho \cup \EhD} 
+t \int_0^t \normsmall{ \cuu^{1/2}\salto{\bomegaI_{\tau}}}_{0,\Eho \cup \EhD}\normsmall{ \cuu^{1/2}\salto{\bomegah} }_{0,\Eho \cup \EhD}\,d\tau\;.
\end{align*}
For $T_3$ we proceed as before and also recall the definition of $\cuu$
\begin{align*}
|T_3|& \lesssim  
\epsilon\norm{\cuu^{1/2} \salto{\bomegah}}_{0,\Eho \cup \EhD}^{2}
+ \frac{1}{\epsilon} \norm{\calA^{1/2} \bcsi^I}_{0,\Th}^2 
+ t \int_0^t  \norm{\calA^{1/2} \bcsi_{\tau}^I}_{0,\Th}   \norm{\cuu^{1/2} \salto{\bomegah}}_{0,\Eho \cup \EhD}\,d\tau\;.
\end{align*}
By collecting all previous estimates, and recalling the definition of the $\normtreMDG{\cdot}$ norm  we obtain 
\begin{multline}\label{eq:stimaMDG}
 \normtreMDG{(\bomegah,\bcsi^h)}^{2} \lesssim \epsilon \left( \normsmall{\cuu^{1/2}\salto{\bomegah} }^2_{0,\Eho \cup \EhD} +\norm{\calD^{1/2} \beps(\bomegah)}_{0,\Th}^2\right) \\
+ \frac{1}{\epsilon} \left( \norm{\calA^{1/2} \bcsi^I}_{0,\Th}^2 +
\normsmall{ \cuu^{1/2}\salto{\bomegaI} }^2_{0,\Eho \cup \EhD}\right)  \\
+ t \int_0^t \left(\normtreMDG{(\bomegaI_{\tau},\bcsi^I_{\tau})} +\|\mathcal{D}^{1/2}\beps(\bomegaI_{\tau})\|_{0,\Th} \right) \left( \normtreMDG{(\bomegah,\bcsi^h)} +\|\mathcal{D}^{1/2}\beps(\bomegah)\|_{0,\Th} \right)\,d\tau\;.
\end{multline}
To conclude we need to bound the norm $\norm{\calD^{1/2} \beps(\bomegah)}_{0,\Th}$ in terms of $\normtreMDG{(\bomegah, \bcsi^h)}$. Here, we use the (second) error equation: by setting $\btau=\mathcal{D} \beps(\bomegah)$ in \eqref{err:01} we have:
\begin{multline*}
\norm{\calD^{1/2} \beps(\bomegah)}_{0,\Th}^{2}= ( \beps(\bomegaI), \mathcal{D} \beps(\bomegah))_{\Th}
- ( \calA (\es) , \mathcal{D} \beps(\bomegah))_{\Th}  
\\ + \langle \cdd\salto{\es}, \salto{\mathcal{D} \beps(\bomegah)}\rangle_{\Eho\cup \EhN} 
 + \langle \salto{\eu}, \av{\mathcal{D} \beps(\bomegah)}_{\delta}\rangle_{\Eho\cup \EhD}.
 \end{multline*}
Then, the Cauchy Schwarz inequality, the definition \eqref{las-ces} of the parameters $\cuu$ and $\cdd$, and reasoning as in the proof of Lemma~\ref{le:aux04}, give
\begin{equation*}
\begin{aligned}
\norm{\calD^{1/2} \beps(\bomegah)}_{0,\Th}
&\lesssim  \|\calD^{1/2} \beps(\bomegaI) \|_{0,\Th}  + \|\mathcal{A}^{1/2}\bcsi^I\|_{0,\Th} 
+ \|\cuu^{1/2}\salto{\bomegaI}\|_{0,\Eho\cup \EhD} 
+ \|\cdd^{1/2}\salto{\bcsi^I}\|_{0,\Eho\cup \EhN},\\
& \phantom{\lesssim} 
+ \|\mathcal{A}^{1/2}\bcsi^h\|_{0,\Th} 
+ \|\cuu^{1/2}\salto{\bomegah}\|_{0,\Eho\cup \EhD} 
+ \|\cdd^{1/2}\salto{\bcsi^h}\|_{0,\Eho\cup \EhN}\\
&\leq   \|\calD^{1/2} \beps(\bomegaI) \|_{0,\Th} + \normtreMDG{(\bomegaI, \bcsi^I)} + \normtreMDG{(\bomegah,\bcsi^h)}.
\end{aligned}
\end{equation*}
Using the previous estimate in \eqref{eq:stimaMDG} and choosing $\epsilon$ so that $1-C\epsilon>0$ gives
\begin{multline*}
 \normtreMDG{(\bomegah,\bcsi^h)}^{2} \lesssim \normtreMDG{(\bomegaI, \bcsi^I)}^2 
+  \|\calD^{1/2} \beps(\bomegaI) \|^2_{0,\Th}\\
+t \int_0^t  \left(\normtreMDG{(\bomegaI_{\tau},\bcsi^I_{\tau})} +\|\mathcal{D}^{1/2}\beps(\bomegaI_{\tau})\|_{0,\Th}\right)\left( \normtreMDG{(\bomegaI, \bcsi^I)}
 +\|\mathcal{D}^{1/2}\beps(\bomegaI)\|_{0,\Th} \right)  \,d\tau
\\
+ t \int_0^t \left(\normtreMDG{(\bomegaI_{\tau},\bcsi^I_{\tau})} +\|\mathcal{D}^{1/2}\beps(\bomegaI_{\tau})\|_{0,\Th}\right)\normtreMDG{(\bomegah,\bcsi^h)} \,d\tau.
\end{multline*}
Using now $t\leq T$ and  Gronwall's Lemma~\ref{Gronwall}, we obtain
\begin{equation*}
 \normtreMDG{(\bomegah,\bcsi^h)}\lesssim 
\sqrt{\mathcal{G}}
+T \int_0^t (\normtreMDG{(\bomegaI_{\tau},\bcsi^I_{\tau})} +\|\mathcal{D}^{1/2}\beps(\bomegaI_{\tau})\|_{0,\Th}) \,d\tau,
\end{equation*}
where
\begin{multline*}
\mathcal{G}=\normtreMDG{(\bomegaI, \bcsi^I)}^2  +  \|\calD^{1/2} \beps(\bomegaI) \|^2_{0,\Th} \\
+ T\int_0^T  
\normtreMDG{(\bomegaI_{\tau},\bcsi^I_{\tau})}^2 +\|\mathcal{D}^{1/2}\beps(\bomegaI_{\tau})\|^2_{0,\Th} + \normtreMDG{(\bomegaI, \bcsi^I)}^2 +\|\mathcal{D}^{1/2}\beps(\bomegaI)\|^2_{0,\Th} 
\,d\tau.
\end{multline*}
The proof is complete using interpolation estimates \eqref{eq:errMDGinterp} and \eqref{eq:errMDGinterp2} and taking the supremum over $t\in (0,T]$.
\end{proof}
\subsection{DG methods for the displacement formulation}
The displacement variational formulation for the continuous solution $\bu$ allowing for discontinuous test functions can be written as:
\begin{equation*}
(\rho  \bu_{tt}, \bv)_{\Th} +( \beps( \bu), \calD \beps(\bv))_{\Th}
 - \langle \av{\calD \beps(\bu)}_\delta, \salto{\bv}\rangle_{\EhoD} 
= ( \bfe , \bv )_{\Th} + ( \bg , \bu_t )_{\calF_h^N}, \quad \forall \bv \in \V
\end{equation*}
where we have used \eqref{per-tens0} and that $\salto{\calD\beps(\bu)} = {\bf 0}$. The error equation is obtained by subtracting  \eqref{dg:ip3} from the above equation, using that $\salto{\bu} = {\bf 0}$ on $F\in \Eh^I \cup \calF_h^D$ and recalling the definition of $a(\cdot,\cdot)$:
\begin{equation}\label{err:0:0}
(\rho \eu_{tt}, \bv)_{\Th} + a(  \eu, \bv) =0 \qquad \forall\, \bv \in \V,
\end{equation}
where $\eu=\bu - \bu^h$.\\

Before stating the main theorem of this section, we observe that with standard arguments it can be proved
\begin{equation}\label{normIPA}
\begin{aligned}
&|a(\bv, \bw)|\lesssim\normA{\bv}\normA{\bw}\lesssim \normtreIP{\bv}\normtreIP{\bw}
&& \forall \bv,\bw \in {\bf H}^2(\Th),
\end{aligned}
\end{equation}
where $\normA{\cdot}$ is defined as in \eqref{norm:ip}.

\begin{theorem}\label{Teo:ErrIP}
Let $\bu$ be the solution of \eqref{eqd:1}--\eqref{eqd:6}, and let $\bu^h \in \V$ be the approximated solution obtained with the SIP($\delta$) method defined in Section~\ref{stability:sec}. 
Assume that the penalty parameter $\czz$ appearing in \eqref{def:penalty_SIP} is large enough.
Then, 
\begin{multline}\label{eq:ErrIP}
\sup_{0<t\leq T}\normtre{\bu(t) -\bu^h(t)}_{\calE,\textrm{IP}} 
\lesssim
h^{k}  \sup_{0<t\leq T} \left(
|\bu(t) |_{k+1,\O}^2 + h^2 | \bu_{\tau}(t) |^2_{k+1,\O}\right)^{1/2}
\\
+h^{k}  \ \int_0^T  \left(
|\bu_{\tau}(\tau) |_{k+1,\O}^2 + h^2 | \bu_{\tau\tau}(\tau) |^2_{k+1,\O}\right)^{1/2} \,d\tau,
\end{multline}
where the hidden constant depends on the polynomial degree $k$, $\Dsup$, $\rsup$, the shape regularity constant of the mesh $\Th$ and depens linearly on the time $T$.
\end{theorem}
%%%%%%%%%%%%%%%%%%%%%%%%%%%%%%%%%%%%%%%%%%%%
\begin{proof}
Writing $\eu = \bomegaI - \bomegah =( \bu - \bu^I)- ( \bu^h - \bu^I)$,
\begin{equation*}
\eu = \bomegaI - \bomegah =( \bu - \bu^I)- ( \bu^h - \bu^I)\;,
\end{equation*}
the triangle inequality gives
\begin{equation*}
\normtreIP{\bu - \bu^h} \leq \normtreIP{ \bomegaI} +\normtreIP{ \bomegah}.
\end{equation*}
It is enough to estimate  $\normtreIP{ \bomegah}$ since
the other term can be bounded by the interpolation estimates \eqref{eq:errIPinterp}.  We follow the stability proof and observe that
the error equation \eqref{err:0:0} becomes
\begin{equation*}
(\rho \bomegah_{tt}, \bv)_{\Th}+a( \bomegah ,\bv) =(\rho \bomegaI_{tt}, \bv)_{\Th} +a(\bomegaI, \bv) \qquad \forall\, \bv \in \V\;.
\end{equation*}
By setting $\bv = \bomegah_t$, we have\begin{equation*}
\frac{1}{2} \frac{d}{dt} \left( \normIP{\bomegah}^2   -2 \langle \salto{\bomegah}, \av{\calD \beps(\bomegah)}_\delta \rangle_{\Eho\, \cup \,\calF_h^D} \right)  
 = (\rho \bomegaI_{tt}, \bomegah_t)_{\Th} +a(\bomegaI ,\bomegah_t)\;.
\end{equation*}
Notice that the right hand side is exactly in the form of \eqref{eq:stab_SIPG} with $\bv_h=\bomegah$ (see proof of Proposition \ref{stab:prim}).  
Integrating in time between $0$ and $t$ and  using that $\bomegah(0) = {\bf 0}$,  we get
\begin{equation*}
\begin{aligned}
\normtreIP{\bomegah}^2   \lesssim&   \int_0^t (\rho \bomegaI_{tt}, \bomegah_{t} )_{\Th} \,d\tau+ \int_0^t  a (\bomegaI, \bomegah_{t} )
\,d\tau\\
 \lesssim&\,
t \int_0^t \norm{\rho^{1/2} \bomegaI_{\tau\tau}}_{0,\Th}
\norm{\rho^{1/2} \bomegah_{\tau}}_{0,\Th}\,d\tau + a (\bomegaI ,\bomegah ) - \int_0^t a (\bomegaI_{t}, \bomegah )\,d\tau,\\
\end{aligned}
\end{equation*}
where in the last step we have used Jensen and Cauchy-Schwarz inequalities for the first term and formula \eqref{eq:ibp_formula} with $\bw=\bomegaI$ and $\bz=\bomegah$, together with $\bomegah(0) = {\bf 0}$ for the second one.
Using continuity of $a(\cdot,\cdot)$, see \eqref{normIPA}, and the arithmetic geometric inequality with $\epsilon>0$, we obtain
\begin{equation*}
\begin{aligned}
(1 -C\epsilon)\normtreIP{\bomegah}^2    \lesssim & \frac{1}{\epsilon} \normtreIP{\bomegaI}^{2} + 
t \int_0^t  \normtreIP{\bomegaI_t} \normtreIP{\bomegah } \,d\tau.
\end{aligned}
\end{equation*}
Choosing $\epsilon$ small enough, using $t\leq T$  and applying Gronwall's Lemma~\ref{Gronwall}, we get
 \begin{equation*}
 \normtreIP{\bomegah}  \lesssim  \normtreIP{\bomegaI} + T  \int_0^t  \normtreIP{\bomegaI_t}   \,d\tau \;.
 \end{equation*}
The proof is concluded by using  the interpolation estimates \eqref{eq:errIPinterp} and taking the supremum over $t \in (0,T]$.
%{\rosso
%\begin{equation*}
%\sup_{0<t \leq T} \normtreIP{\bomegaI} \lesssim \sup_{0 < t \leq T} h^{k}\left( |\bu |_{k+1,\O}^2 + h^2| \bu_t|_{k+1,\O}^2\right)^{1/2}
%\end{equation*}
%\begin{equation*}
%\int_0^t \norm{\rho^{1/2} \bomegaI_{\tau\tau}}_{0,\Th} \,d\tau \lesssim h^{k} \int_0^t h |\bu_{\tau\tau} |_{k+1,\O} \,d\tau
%\end{equation*}
%\begin{equation*}
%\int_0^t  \normtreIP{\bomegaI_t}  \,d\tau \lesssim h^{k} \int_0^t \left( |\bu_\tau |_{k+1,\O}^2 + h^2| \bu_{\tau\tau}|_{k+1,\O}^2\right)^{1/2} \, d\tau
%\end{equation*}
%}

 \end{proof}

\section{Numerical results}\label{num_res}
\label{sec:numerics}
To conclude our analysis we present some numerical results obtained with the DG spectral element code SPEED (\verb+http://mox.polimi.it/it/progetti/speed+), cf. \cite{MazzieriStupazziniGuidottiSmerzini_2013}, where the fully discrete solution is recovered by coupling our semidiscrete formulation with the second order accurate explicit leap-frog time integration scheme.

We start showing a test case confirming the (optimal) theoretical convergence rates shown in  Section~\ref{sec:error_estimates}. For brevity, we only report the results obtained with the SIP($\delta$) method  with $\delta=1/2$, cf. Section~\ref{sec:DGprimal}, but similar results can be obtained with the NIP and IIP methods in \eqref{dg:ip3}. \\

We solve a wave propagation problem in $\Omega = (0,1)^3$, set the Lam\'e parameters $\lambda,\mu$  and the mass density $\rho$ equal to $1$, and choose $\bff$ such that problem \eqref{eqd:1}--\eqref{eqd:6} features the exact solution 
\begin{equation*}
\bu(\bx,t) = \sin(3 \pi t)\left[ \begin{array}{c} -\sin^2(\pi x)\sin(2\pi y)\sin(2\pi z) \\ \phantom{-}\sin(2\pi x)\sin^2(\pi y)\sin(2\pi z) \\ \phantom{-}\sin(2\pi x)\sin(2 \pi y)\sin^2(\pi z) \end{array} \right].
\end{equation*}
The Dirichlet boundary conditions on the whole $\partial\Omega$, the initial displacement $\bu_0$, and initial velocity $\bu_1$ are set accordingly.
We consider a Cartesian decomposition of the domain $\Omega$ and define four levels of refinements, corresponding to mesh sizes $h=0.5,0.25,0.125,0.0625$ (resp. $h=0.25,0.125,0.0625,0.03125$) for a polynomial approximation degree $k \geq 2$ (resp. $k=1$).
 The simulations have been carried out for a time interval $[0,10]$,  using a time step $\Delta t = 1 \cdot 10^{-5}$. Notice that a so small time step together with the smooth temporal behavior of $\bu$, guarantee that the temporal component of the error does not affect the spatial component. The IP energy norm has been evaluated at each discrete time $t_n = t_0 + n\Delta t$, for $n=1,...,10^6$ so that 
the norm $\max_{0<t_n\leq T} \| (\bu - \bu^h)(t_n) \|_{\calE,\textrm{IP}}$ represents a good approximation of the norm in Theorem~\ref{Teo:ErrIP}.
%%%%%%%%%%%%%%%%%%
\begin{figure}[!htbp]
\centering
\includegraphics[width=0.65\columnwidth]{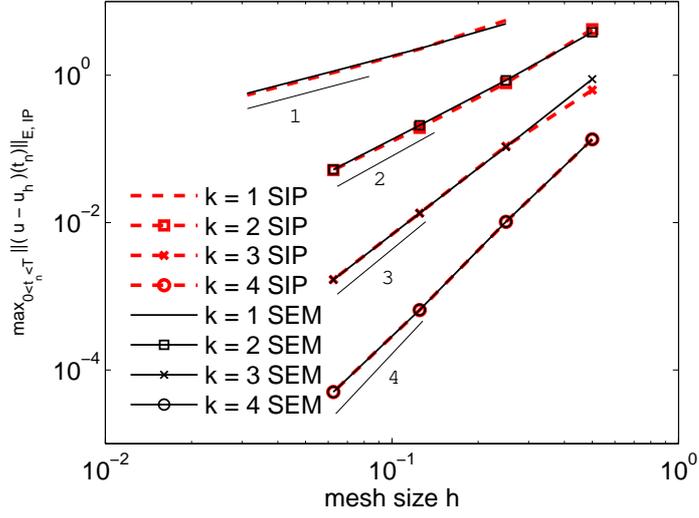}
\caption{Computed energy error versus the mesh size $h$, for different polynomial approximation degrees $k=1,2,3,4$ (loglog scale).}\label{Figure1}
\end{figure}
%%%%%%%%%%%%%%%%%%
%
Since high order spatial approximation of elastodynamics problems have been previously addressed in the context of spectral or spectral element methods (with continuous polynomial matching across interelements), cf. \cite{QuarteroniZampieri_1992,CividiniQuarteroniZampieri_1993,
Faccioli01081996,FaMaPaQu97,KoVi98}, for completeness, we therefore compare the numerical results obtained with the SIP($\delta$) method to the analogous ones obtained with the spectral element method (SEM). For the SE approximation the spatial error has been evaluated using the energy norm  defined in \eqref{norm:ip}, obviously neglecting the last term.    

%%%%%%%%%%%%%%%%%%
\begin{table}[!htbp]
\caption{Computed convergence rates for different polynomial approximation degrees $k=1,2,3,4$.}\label{Table1}
\centering
\begin{tabular}{|c|c|c|c|c|}
\hline
 & \multicolumn{4}{c|}{k} \\
\hline
method    & 1 & 2 & 3 & 4 \\
\hline
 SIP($\delta$) &  1.1212 &  2.1157 & 2.8478 & 3.7973\\
\hline
 SEM &  0.9492 &  2.0622 & 3.0135 &  3.7973\\
\hline
\end{tabular}
\end{table}
%%%%%%%%%%%%%%%%%%
In Figure~\ref{Figure1} we show the  computed error as a function of the mesh size $h$ for different polynomial approximation degrees $k=1,2,3,4$;
the corresponding computed convergence rates are reported in Table \ref{Table1}. 
The numerical results confirm the theoretical results proved in Theorem~\ref{Teo:ErrIP} and demonstrate once again the $h-$optimality of DG discretizations.\\ 

We now investigate the stability of the NIP method \eqref{dg:ip3}, (with $\theta=1$ in \eqref{eq:defa}) and the corresponding modification given in \eqref{extraAA} (originally introduced in  \cite{RiWh03}).
In Figure~\ref{Figure2} we show the computed error as a function of the time step $\Delta t$ for a fixed mesh size $h=0.125$ varying the  polynomial approximation degree $k=1,2,3$. 
%%%%%%%%%%%%%%%%%%%%%%%%%%%%%%%%%%%%%%
\begin{figure}
\centering\includegraphics[width=0.65\columnwidth]{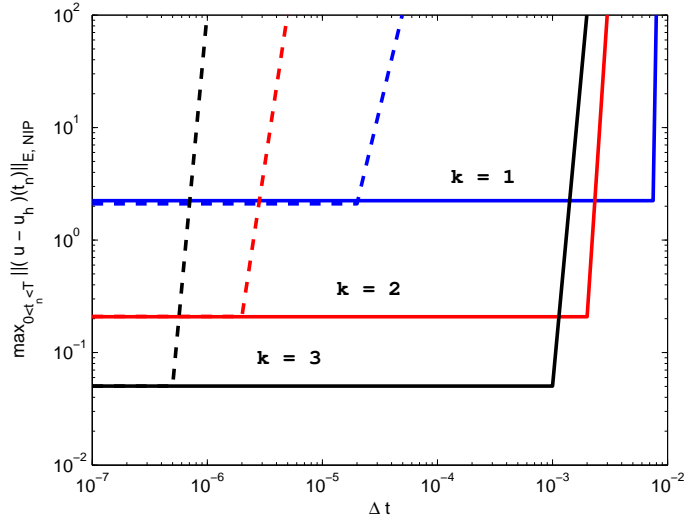}
\caption{Computed energy error versus the time step $\Delta t$ for different polynomial approximation degrees $k=1,2,3$ and $h=0.125$ (loglog scale). NIP method \eqref{dg:ip3} (--) vs  modified NIP \eqref{extraAA}  \cite{RiWh03} method (- -).}\label{Figure2}
\end{figure}
%%%%%%%%%%%%%%%%%%%%%%%%%%%%%%%%%%%%%%
Similar results where obtained for  different  mesh sizes $h$. As it can be seen from the graphics, the presence of the additional stabilization term imposes a much severe restriction  on the time step size required to guarantee stability in practice, than for the original IP method. So even if the extra term is helpful for the theoretical analysis of the method \eqref{extraAA}, it needs to be handled extremely carefully in the numerical simulations, in order to guarantee the stability in practice, and it seems to impose a much stringent condition on the time step $\Delta t$.  From these experiments,  we infer that the efficiency of the modified scheme \eqref{extraAA} seems to be lower than that of the plain IP \eqref{dg:ip3}.
%%%%%%%%%%%%%%%%%%%%%%%%%%%%%%%%%%%%%%
\begin{figure}
\centering
\includegraphics[width=0.65\columnwidth]{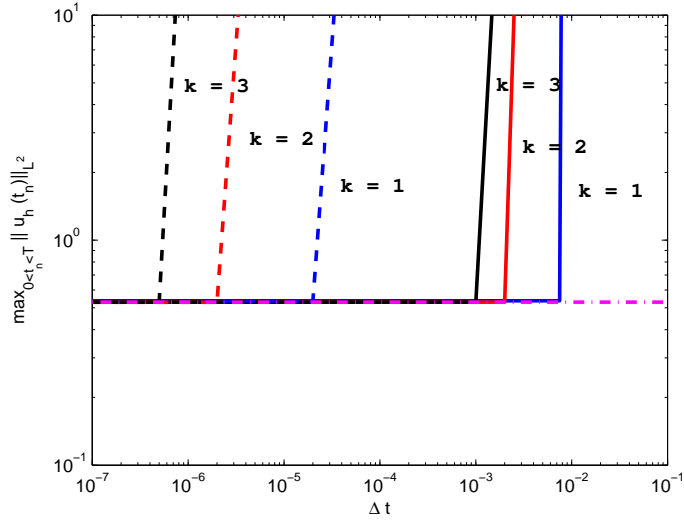}
\caption{Computed $L^\infty(L^2)$ norm of $\bu_h$ versus the time step $\Delta t$, for different polynomial approximation degrees $k=1,2,3$ (loglog scale) and $h=0.125$. NIP method \eqref{dg:ip3} (--) vs  modified NIP\eqref{extraAA} -\cite{RiWh03} method (- -). In the plot is also reported the $L^\infty(L^2)$-norm  of the exact solution $\bu$ (.-).}\label{Figure3}
\end{figure}
%%%%%%%%%%%%%%%%%%%%%%%%%%%%%%%%%%%%%%

%%%%%%%%%%%%%%%%%%%%%%%%%%%%%%%%%%%%%%
\begin{figure}[!htpb]
\centering
\includegraphics[width=0.65\columnwidth]{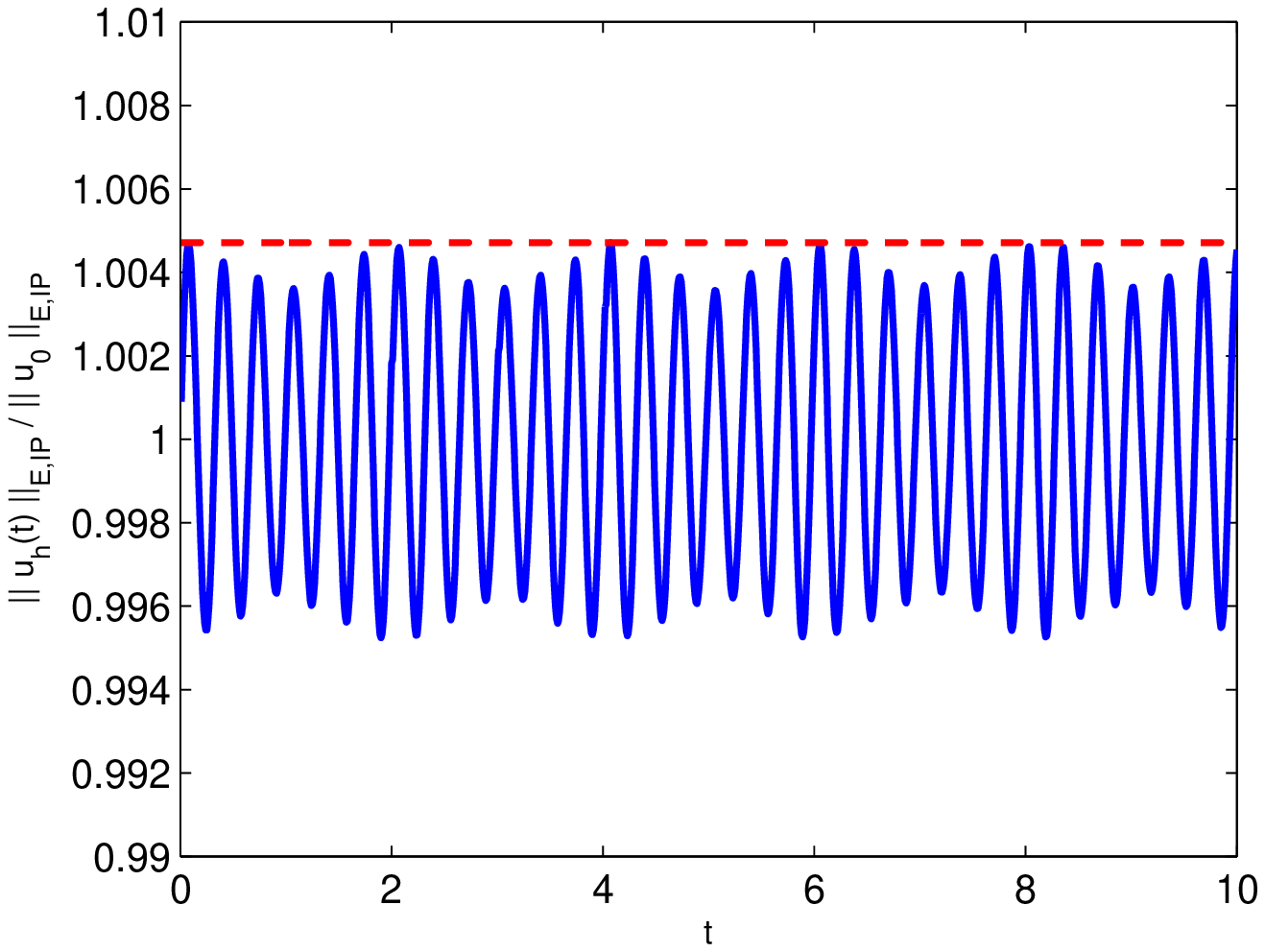}
\caption{Solid line: computed ratio $\displaystyle {\| \bu^h(t_n) \|_{\calE,\textrm{IP}}}/{\| \bu_0^h \|_{\calE,\textrm{IP}}}$ as a function of time $t$ for $h=0.125$, $k=3$ and  $\Delta t = 1 \cdot 10^{-4}$. 
Dashed line: upper bound $C=\displaystyle \max_{0< t \leq 10} {\| \bu^h(t) \|_{\calE,\textrm{IP}}}/{\| \bu_0^h \|_{\calE,\textrm{IP}}}$.}\label{Figure4}
\end{figure}
%%%%%%%%%%%%%%%%%%%%%%%%%%%%%%%%%%%%%%
Finally, to conclude the analysis, we consider the homogeneous elastic domain $\Omega = (0,1)^3$, with   $\lambda,\mu,\rho = 1$ not excited by external loads, i.e., ${\bf f} = {\bf 0}$.  We set homogeneous Dirichlet boundary conditions and initial conditions 
\begin{equation*}
\bu_0(\bx) = {\bf 0}, \qquad \bu_1(\bx) =  \left[ \begin{array}{c} - 3 \pi \sin^2(\pi x)\sin(2\pi y)\sin(2\pi z) \\ \phantom{-}3 \pi \sin(2\pi x)\sin^2(\pi y)\sin(2\pi z) \\ \phantom{-}3 \pi \sin(2\pi x)\sin(2 \pi y)\sin^2(\pi z) \end{array} \right].
\end{equation*}
for the displacement and the velocity field, respectively. \\For the sake of brevity we report only the results obtained with the SIP($\delta$) approximation, for  $\delta=1/2$,  fixing the mesh size $h=0.125$, the polynomial degree $k=3$, and the time step $\Delta t = 1 \cdot 10^{-4}$ for a time interval $[0,10]$. 
In Figure~\ref{Figure4} we report the ratio ${\| \bu^h(t_n) \|_{\calE,\textrm{IP}}}/{\| \bu_0^h \|_{\calE,\textrm{IP}}}$   as a function of  time $t$. It is easy to see that the above ratio is bounded for all $t \in [0,10]$, verifying the property \textit{(i)} of Proposition~\ref{stab:prim}, i.e.,  ${\| \bu^h(t_n) \|_{\calE,\textrm{IP}}}/{\| \bu_0^h \|_{\calE,\textrm{IP}}} \leq  C$.  
\FloatBarrier

\section*{Acknowledgments}
Part of this work has been completed during several visits of the second author to the 
IMATI-CNR of Pavia. She is grateful to the IMATI for the kind
hospitality.  The second author was partially supported by MEC through
the project MTM2011-27739-C04-04.

\appendix
%%%%%%%%%%%%%%
%%%%%%%%%%%%%%
%%%%%%%%%%%%%%
%%%%%%%%%%%%%%
%%%%%%%%%%%%%%
%%%%%%%%%%%%%%
%%%%%%%%%%%%%%
%%%%%%%%%%%%%%
\section{Proof of Lemma~\ref{le:0}}
\label{sec:proof-stability}
\begin{proof}[Proof of Lemma~\ref{le:0}] The proof of the first estimate is achieved in three steps.\\
\noindent {\bf Step I.} We consider the weak formulation \eqref{weak:1}---\eqref{weak:2} and take as test function $\bv=\bu_t  \in {\bf H}^{1}_{0,\Gamma_D}(\O)$ in \eqref{weak:1}, getting
\begin{equation}\label{cont_u} 
( \rho \bu_{tt} , \bu_t )_{\O} + ( \bsig , \beps(\bu_t) )_{\O}  = ( \bfe , \bu_t )_{\O} + ( \bg , \bu_t )_{\Gamma_N}.   
\end{equation}
Notice that we have used implicitly the time derivative of the Dirichlet boundary condition, expressed in equation \eqref{eqd:3}.\newline
\noindent {\bf Step II.} We differentiate with respect to time the constitutive equation \eqref{eqd:2} 
\begin{equation}\label{2dt}
 \calA \bsig_t-\beps(\bu_t)=0  \qquad  \textrm{in} \; \Omega \times (0,T], 
\end{equation}
and we consider its weak formulation obtained taking as a test function $\bsig\in \vect{\mathcal{L}}^{2}(\O)$:
\begin{equation*} 
\displaystyle ( \calA \bsig_t , \bsig )_{\O} - ( \beps(\bu_t) , \bsig )_{\O} = 0.
\end{equation*}
\noindent {\bf Step III.} Summing up  the above equation  and \eqref{cont_u}, we have
\begin{equation*}
 ( \rho \bu_{tt} ,  \bu_t )_{\O} + ( \calA  \bsig_t , \bsig )_{\O} =  ( \bfe , \bu_t )_{\O} + (\bg , \bu_t )_{\Gamma_N},   
\end{equation*}
or equivalently
\begin{equation*}
\frac{1}{2} \frac{d}{dt} \left( \| \rho^{1/2}  \bu_{t} \|^2_{0,\O}  + \| \calA^{1/2} \bsig \|_{0,\O}^2 \right)  =  ( \bfe , \bu_t )_{\O} + ( \bg , \bu_t )_{\Gamma_N}.   
\end{equation*}
Integration in time of the above identity, taking into account the definition of the energy norm \eqref{norm:E}, yields to  
\begin{equation}\label{stima_en1}
 \frac{1}{2}\| (\bu,\bsig) \|_{\calE}^2 =  \frac{1}{2} \| (\bu_0,\bsig_0) \|_{\calE}^2 + \int_0^t  ( \bfe ,  \bu_\tau)_{\O} \, d\tau + \int_0^t ( \bg , \bu_{\tau} )_{\Gamma_N} \, d\tau .  
\end{equation}
If  $\bff=\bg= {\bf 0}$ the result trivially follows.
Otherwise, to obtain the thesis we need to estimate the last two integrals above. For the first one, the Cauchy-Schwarz inequality gives
\begin{equation}\label{st:00}
\left|\int_0^t  ( \bfe ,  \bu_\tau)_{\O} \,d\tau\right| 
\leq   \int_0^t   \rinf^{-1/2} \| \bfe \|_{0,\O}\|\rho^{1/2} \bu_\tau \|_{0,\O} \, d\tau 
\leq  \int_0^t   \rinf^{-1/2} \| \bfe \|_{0,\O}\norm{(\bu ,\bsig)}_{\calE} \, d\tau,
\end{equation}
where we have also used the lower bound on the material density \eqref{material:0}. 
Before dealing with the last integral in \eqref{stima_en1}, we preliminarily observe that the trace theorem  (cf. \cite{AdamsFournier2003}) together with Korn's inequality (cf. \cite{DuvautLions_1976}, for example) and the lower bound in \eqref{cota:D} gives:
\begin{equation}\label{eq:est-trace}
\|\bu\|_{0,\Gamma_N}^2 
 \leq \|\bu\|_{1/2,\Gamma_N} ^2
\lesssim \| \beps(\bu)\|_{0,\O}  ^2
=  \|\calA\bsig\|_{0,\O}^2
\leq \Dinf^{-1}  \|(\bu ,\bsig)\|_{\calE}^2.
\end{equation}
We rewrite the last integral in \eqref{stima_en1} by using Fubini's theorem together with integration by parts (with respect to time)
\begin{align*}
\int_0^t ( \bg , \bu_{\tau} )_{\Gamma_N}\, d\tau &= \int_{\Gamma_N} \int_0^t   \bg  \cdot \bu_{\tau} \,  d\tau \,ds \\
&=\int_{\Gamma_N} \left[\bg (0) \cdot\bu(0) - \bg (t) \cdot\bu(t) \right]\,ds- \int_{\Gamma_N} \int_0^t  \bg_{\tau}  \cdot \bu \, d\tau \,ds  \\
&=(\bg_0,\bu_0)_{0,\Gamma_N} - ( \bg , \bu)_{0,\Gamma_N} -  \int_0^t  (\bg_{\tau} , \bu)_{0,\Gamma_N} \, d\tau.
\end{align*}
The above identity together with the Cauchy-Schwarz and the  arithmetic-geometric inequalities lead to
\begin{multline*}
\left| \int_0^t ( \bg , \bu_{\tau} )_{\Gamma_N} d\tau \right| 
\leq \frac{1}{2} \|\bg_0\|_{0,\Gamma_N}^{2} 
+ \frac{1}{2} \|\bu_0\|_{0,\Gamma_N}^{2} 
+\epsilon \|\bg\|_{0,\Gamma_N}^{2} 
+\frac{1}{\epsilon} \|\bu\|_{0,\Gamma_N}^{2}
+ \int_0^t \|\bg_{\tau}\|_{0,\Gamma_N} \|\bu\|_{0,\Gamma_N}\, d\tau,
\end{multline*}
where $\epsilon>0$ is still at our disposal. Using 
\eqref{eq:est-trace}  (also for the initial data, which is assumed to be compatible) we finally obtain 
\begin{multline*}
\left| \int_0^t ( \bg , \bu_{\tau} )_{\Gamma_N} d\tau \right|  
\lesssim 
\frac{1}{2} \|\bg_0\|_{0,\Gamma_N}^{2}
+ \frac{\Dinf^{-1}}{2}   \|(\bu_0 ,\bsig_0)\|^{2}_{\calE} 
+  \epsilon \Dinf^{-1} \|\bg\|_{0,\Gamma_N}^{2}  \\
+ \frac{\Dinf^{-1}}{\epsilon}   \|(\bu ,\bsig)\|^{2}_{\calE} 
+   \int_0^t  \|\bg_{\tau}\|_{0,\Gamma_N} \|(\bu ,\bsig)\|_{\calE} d\tau.
\end{multline*}
Substituting the above estimate together with \eqref{st:00} in the identity \eqref{stima_en1}, we  get
\begin{multline*}
\left(\frac{1}{2} -C\frac{\Dinf^{-1}}{\epsilon}\right)\| (\bu,\bsig) \|_{\calE}^2 \lesssim
 \frac{1}{2}  \|\bg_0\|_{0,\Gamma_N}^{2}
+\frac{(\Dinf^{-1}+1)}{2} \|(\bu_0 ,\bsig_0)\|^{2}_{\calE} \\
+ \epsilon \Dinf^{-1}\|\bg\|_{0,\Gamma_N}^{2}  
+  \int_0^t  \left( \rinf^{-1/2} \| \bfe \|_{0,\O} + \norm{\bg_{\tau}}_{0,\Gamma_N} \right)\norm{(\bu ,\bsig)}_{\calE} \, d\tau.
\end{multline*}
Choosing $\epsilon$ so that $1/2 -C\Dinf^{-1}\epsilon^{-1}>0$ we  have
\begin{multline*}
\| (\bu,\bsig) \|_{\calE}^2 \lesssim
\frac{1}{2}  \|\bg_0\|_{0,\Gamma_N}^{2}
+\frac{(\Dinf^{-1}+1)}{2} \|(\bu_0 ,\bsig_0)\|^{2}_{\calE} \\
+  \Dinf^{-1}\|\bg\|_{0,\Gamma_N}^{2}  
+  \int_0^t  \left( \rinf^{-1/2} \| \bfe \|_{0,\O} + \norm{\bg_{\tau}}_{0,\Gamma_N} \right)\norm{(\bu ,\bsig)}_{\calE} \, d\tau.
\end{multline*}
The proof is complete applying Gronwall's lemma
\begin{equation*}
\| (\bu,\bsig) \|_{\calE} \lesssim
\sqrt{\mathcal{G}} + \frac{1}{2} \int_0^t  \left( \rinf^{-1/2} \| \bfe \|_{0,\O} + \norm{\bg_{\tau}}_{0,\Gamma_N} \right) \, d\tau,
\end{equation*}
with 
\begin{equation*} 
\mathcal{G}=
 \frac{1}{2}  \|\bg_0\|_{0,\Gamma_N}^{2}
+\frac{(\Dinf^{-1}+1)}{2}\|(\bu_0 ,\bsig_0)\|^{2}_{\calE} \\
+  \Dinf^{-1} \sup_{0 < t \leq T} \|\bg\|_{0,\Gamma_N}^{2}.
\end{equation*}
 \end{proof}
%%%%%%%%%%
\section{Proof of Lemma~\ref{le:aux00} and Lemma~\ref{le:aux04}}\label{sec:ProofAuxiliary}
In this appendix we collect the proofs of the auxiliary Lemmas~\ref{le:aux00} and \ref{le:aux04}, used in the stability analysis.
\begin{proof}[Proof of Lemma~\ref{le:aux00}]
The proof goes along the same lines as in the continuous case with subtle modifications to obtain bounds independent of $h$. 

Estimate \eqref{aux:00} follows from  the Cauchy-Schwarz  inequality  together with the lower bound of the mass density \eqref{material:0}. 
To show  estimate \eqref{aux:02}, we proceed similarly to get
\begin{equation}\label{rr:00}
\left|  \int_0^t  \langle \cdd \bg_\tau, \bsig^{h}\,\n\rangle_{\EhN} \, d\tau\right| \leq   \Dinf^{-1/2} \int_0^t  \| \bg_\tau \|_{1/2, \Gamma_N} \|\cdd^{1/2} \bsig^{h}\n\|_{0,\EhN}\, d\tau \;.
\end{equation}
Next, we notice that for each $t\in [0,T]$, the map $\bg(t)$ belongs to ${\bf H}^{1/2}(\Gamma_N)$. The inverse trace theorem \cite{AdamsFournier2003}  guarantees that the trace operator has a continuous right inverse operator, say $\mathfrak{T}: {\bf H}^{1/2}(\Gamma_N) \longrightarrow {\bf H}^{1}(\Omega)$. Hence, taking into account the scaling of the parameter $\cdd$  and using the trace inequality \eqref{trace:0} we have
\begin{equation*}
\begin{aligned}
\|\cdd^{1/2}\bg\|_{0,F}^{2}=c_2 h_F k^{-2}  \av{\calD}^{-1} \|\bg\|_{0,F}^{2} \lesssim  c_2 k^{-2} \Dinf^{-1}\|\bg \|_{1,\K}^{2}  
&& \forall\, F\in \EhN,  F \subset \partial \K\;,
\end{aligned}
\end{equation*}
where, with an abuse of notation, we have denoted by $\bg =\mathfrak{T}\bg$ the extension of $\bg$. Summing over all $F\in \EhN$ and using the continuity of the operator $\mathfrak{T}$ we get
\begin{equation}\label{eq:pongo-ora}
 \norm{\cdd^{1/2} \bg_\tau}^{2}_{0,\EhN} 
\lesssim  \sum_{\K \in \Th }
\Dinf^{-1}\|\bg_\tau \|_{1,\K}^{2}  =  \Dinf^{-1}\|\bg_\tau\|^{2}_{1,\O} \lesssim  \Dinf^{-1}\| \bg_\tau \|_{1/2, \Gamma_N}^{2}. 
\end{equation}
Substitution of the above estimate in \eqref{rr:00} gives \eqref{aux:02}.\\
To prove \eqref{aux:01}, we use integration by parts formula \eqref{eq:ibp_formula} with $\bw=\bg$ and $\bz=\bu^h$, together with triangle and Jensen' inequality to get
\begin{equation}\label{formula-00}
\left| \int_0^t \langle\bg , \bu^{h}_{\tau}\rangle_{\EhN} \, d\tau \right|
\leq \left|\langle\bg_{0},\bu^{h}_0\rangle_{\EhN}\right|   +\left|\langle \bg , \bu^{h}\rangle_{\EhN}\right|  + t\int_0^t  \left| \langle\bg_{\tau} , \bu^{h}\rangle_{\EhN}\right|\,d\tau.
\end{equation}
Therefore, we only need to estimate the  inner product $|\langle \bg , \bu^{h}\rangle_{\EhN}|$, where the first argument could be either $\bg_0$, $\bg$ or $\bg_{\tau}$.
Applying H\"older's inequality, the trace inequality \eqref{trace:0p} and inequality  \eqref{ine:00} with $\omega=F\in \EhN$ gives
\begin{equation*}
\begin{aligned}
\left|\int_{F} \bg, \bu^{h} ds 
\right| &
\leq
\|\bg \|_{L^{q}(F)} \|\bu^{h}\|_{L^{p}(F)} 
\lesssim
\|\bg\|_{L^{q}(F)} h^{-1/p} \| \bu^{h}\|_{\vect{W}^{1,p}(\K)} \\
  &\lesssim
 \|\bg\|_{L^{q}(F)} h^{-1/p}  h^{d(\frac{1}{p}-\frac{1}{2})}\| \bu^{h}\|_{1,\K}
= \|\bg\|_{L^{q}(F)}  h^{\frac{2d-2 -dp}{2p}} \| \bu^{h}\|_{1,\K}\;,
   \end{aligned}
 \end{equation*}
where, for any $F\in \EhN$,  $K$ is the only element in $\calT_h$ such that $F \subset \partial K$.
Setting now $p=(2d-2)/d$ (whose conjugate is $q=\frac{(2d-2)}{(d-2)}$) the above inequality becomes
\begin{equation}\label{eqo:0}
\left|\int_{F} \bg, \bu^{h} ds 
\right|  \lesssim   \|\bg\|_{L^{q}(F)}   \| \bu^{h}\|_{1,\K}. \end{equation}
Notice that $q=\infty$ for $d=2$ and $q=4$ for $d=3$.
Using that $F$ is a $d-1$ dimensional element and using the continuity of the Sobolev embedding $H^{1}(F) \longrightarrow L^{q}(F)$ \cite{AdamsFournier2003}, we have
  \begin{equation*}
\begin{aligned}
& \|\bg\|_{\bL^{q}(F)} \lesssim \|\bg\|_{1, F} 
&&  \forall\, \bg \in {\bf H}^{1}(F), 
&& q=\frac{(2d-2)}{(d-2)}.
\end{aligned}
  \end{equation*}
 Substituting  the above bound in \eqref{eqo:0} and summing over all faces $F \in \calF_h^N$, gives 
 \begin{equation}\label{eq:estimate:gtilde}
 \begin{aligned}
 \left|\langle \bg , \bu^{h}\rangle_{\EhN}\right| \lesssim&  \|\bg\|_{1,\Gamma_N}\left( \|\bu^{h}\|_{0,\Th}^{2} + \abs{\bu^{h}}_{1,\Th}^{2} \right)^{1/2}.
 \end{aligned}
\end{equation}
Applying the discrete Poincar\'e and Korn inequalities  \cite{Brenner_2003,Br04}, and the bound in \eqref{cota:D}, we have 
\begin{align*}
\|\bu^{h}\|_{0,\Th}^{2} + |\bu^{h}|_{1,\Th}^{2} 
&\lesssim \|\beps{(\bu^{h})}\|_{0,\Th}^{2} + \sum_{F\in \Eho\cup\EhD} \|h_F^{-1/2}\salto{\bu^{h}}\|_{0,F}^{2} \\
  &\lesssim \Dinf^{-1}  \left(\|\mathcal{D}^{1/2}\beps{(\bu^{h})}\|_{0,\Th}^{2} + \sum_{F\in \Eho\cup\EhD} \|h_F^{-1/2}\av{\mathcal{D}}^{1/2}\salto{\bu^{h}}\|_{0,F}^{2}  \right)\\
&= \Dinf^{-1} \normA{\bu^{h}}^{2}.
 \end{align*}
Finally, substituting the above estimate in \eqref{eq:estimate:gtilde} yields
\begin{equation*}
 \begin{aligned}
 \left|\langle \bg , \bu^{h}\rangle_{\EhN}\right| \lesssim&     \|\bg\|_{1,\Gamma_N}\Dinf^{-1}  \normA{\bu^{h}}.
 \end{aligned}
\end{equation*}
Applying now the above estimate to each term in \eqref{formula-00}, we finally get
\begin{equation*}
 \begin{aligned}
 \abs{\langle\bg_{0} , \bu^{h}_0\rangle_{\EhN}}
&\lesssim \Dinf^{-1} 
 \norm{\bg_{0}}_{1,\Gamma_N} \normA{\bu^{h}_0}, &&\\
\int_0^t  \abs{\langle\bg_{\tau} , \bu^{h}\rangle_{\EhN}} \,d\tau
\lesssim & \int_0^t \Dinf^{-1/2} \| \bg_{\tau}\|_{1,\Gamma_N} \normA{\bu^{h}} \,d\tau. &&\\
\abs{\langle\bg (t) , \bu^{h}(t)\rangle_{\EhN}} 
&\lesssim 
\frac{ \Dinf^{-1}}{\epsilon}
\norm{\bg (t)}_{1,\Gamma_N}^2 
+ \epsilon  \normA{\bu^{h}(t)}^2 , &&
\end{aligned}
\end{equation*}
where for the last term we have also used the arithmetic geometric inequality with $\epsilon >0$. Substitution of the above estimates into \eqref{formula-00} completes the proof.
%\end{proof}
\end{proof}
%%%%%%%%%%%%%%%%%%%%%%%%%%%%%%%%%%%%%%%%%%%%%%%%%%%%%%%%%%%%%%%%%%%%%%%
%
\begin{proof}[Proof of Lemma~ \ref{le:aux04}]
We start rewriting the second equation  in \eqref{dg:mixform} with $\btau =\calD\beps(\bu^h)$
\begin{equation}\label{eq:lemmamisti}
\begin{aligned}
\|\mathcal{D}^{1/2}\beps{(\bu^{h})}\|_{0,\Th}^{2}&=
( \beps(\bu^h), \calD\beps(\bu^h))_{\Th}= ( \calA \bsig^h , \calD\beps(\bu^h))_{\Th} 
+\langle \cdd\salto{\bsig^{h}}, \salto{\calD\beps(\bu^h)}\rangle_{\Eho} \\
&\qquad +\langle \cdd (\bsig^{h}\bn -\bg), \calD\beps(\bu^h) \,\n\rangle_{\EhN}
+ \langle \salto{\bu^h}, \av{\calD\beps(\bu^h)}\rangle_{\Eho\cup \EhD}\\
&\qquad-\langle \av{\bu^{h}}_{(1-\delta)}-\av{\bu^h}, \salto{\calD\beps(\bu^h)} \rangle_{\Eho}, 
\end{aligned}
\end{equation}
Prior to estimate all terms on the right-hand side above, we note that Agmon's \eqref{agmon} and inverse inequalities \eqref{eq:inverse}, and the definition of $\cdd$ give
\begin{equation}\label{estima222}
  \| \cdd^{1/2} \salto{\calD\beps(\bu^h)} \|_{0,F}  
\lesssim \| \calD^{1/2}\beps(\bu^h) \|_{0,K},\qquad  \| \cdd^{1/2}  \av{\calD\beps(\bu^h)}_{\delta} \|_{0,F}  
\lesssim \| \calD^{1/2}\beps(\bu^h) \|_{0,K}\;.
 \end{equation}
Now, using  Cauchy-Schwarz inequality  and the first estimate above, the first three terms  in \eqref{eq:lemmamisti} can be bounded by
\begin{eqnarray*}
\left| ( \calA \bsig^h , \calD\beps(\bu^h))_{\Th} \right| & \leq & \|\mathcal{A}^{1/2}\bsig^{h} \|_{0,\Th}  \|\mathcal{D}^{1/2}\beps{(\bu^{h})}\|_{0,\Th}, \\ 
 \left| \langle \cdd\salto{\bsig^{h}}, \salto{\calD\beps(\bu^h)}\rangle_{\Eho} \right| & \lesssim & \| \cdd^{1/2} \salto{\bsig^h}\|_{0,\Eho} \| \calD^{1/2}\beps(\bu^h) \|_{0,K}  \\ 
\left| \langle \cdd (\bsig^{h}\bn -\bg), \calD\beps(\bu^h) \,\n\rangle_{\EhN} \right| & \lesssim &
\left( \| \cdd^{1/2} \salto{\bsig^h} \|_{0,\EhN} + \| \cdd^{1/2} \bg \|_{0,\EhN} \right) \| \calD^{1/2}\beps(\bu^h) \|_{0,K} .
\end{eqnarray*}
To estimate the last two terms in \eqref{eq:lemmamisti}, notice that $\cuu \cdd =O(1)$ since,
\begin{equation*}
\cuu^{-1} = \left(c_1  h_F^{-1} k^{2} \av{\calD}\right)^{-1} = (c_1 c_2)^{-1} c_2 h_F k^{-2} \av{\calD}^{-1} = (c_1 c_2)^{-1} \cdd.
\end{equation*}
Then, the Cauchy Schwarz inequality and \eqref{estima222} give for the fourth term
\begin{eqnarray}
\left| \langle \salto{\bu^h}, \av{\calD\beps(\bu^h)}\rangle_{\Eho\cup \EhD} \right| 
&\lesssim & \| \cuu^{1/2} \salto{\bu^h}\|_{0,\Eho\cup \EhD} \; \| \cdd^{1/2} \av{\calD\beps(\bu^h)} \|_{0,\Eho\cup \EhD}\nonumber\\
&\lesssim & \| \cuu^{1/2} \salto{\bu^h}\|_{0,\Eho\cup \EhD} \; \| \calD^{1/2}\beps(\bu^h) \|_{0,\Omega}.\label{caspa200}
 \end{eqnarray}
Analogously, the last term can be estimated using identity \eqref{eq:media-mediad} and \eqref{estima222}
\begin{equation*}
\left| - \langle \av{\bu^{h}}_{(1-\delta)}-\av{\bu^h}, \salto{\calD\beps(\bu^h)} \rangle_{\Eho} \right|  \lesssim  
 \| \cuu^{1/2} \salto{\bu^{h}} \|_{0,\Eho} \| \calD^{1/2}\beps(\bu^h) \|_{0,\Omega}.
\end{equation*}
Finally, 
substituting all the estimates into
 \eqref{eq:lemmamisti}  we obtain 
\begin{align*}
\|\mathcal{D}^{1/2}\beps{(\bu^{h})}\|_{0,\Th}
\lesssim &
\|\mathcal{A}^{1/2}\bsig^{h} \|_{0,\Th}
+ 
\| \cuu^{1/2} \salto{\bu^h}\|_{0,\Eho}
+
 \| \cdd^{1/2} \salto{\bsig^h} \|_{0,\EhN}
+
 \| \cdd^{1/2} \bg \|_{0, \EhN},
\end{align*}
The proof is then  concluded by arguing as in the proof of \eqref{aux:02} in Lemma \ref{le:aux00} (using estimate \eqref{eq:pongo-ora})
\begin{equation*}
 \norm{\cdd^{1/2} \bg}^{2}_{0,\EhN} 
\lesssim \Dinf^{-1}\| \bg \|_{1/2, \Gamma_N}^{2}.
\end{equation*}
\end{proof}

%\bibliographystyle{abbrv}
%\bibliography{latexbi}
%%%%%%%%%%%%%%

%%%%%%%%%%
%%%%%%%%%%
\end{document}